\documentclass{amsart}

\title[Tensor categories of affine Lie algebras beyond admissible levels]{{\bf Tensor categories of affine Lie algebras\\ beyond admissible levels}}
\thanks{We would like to thank Drazen Adamovi$\acute{\mbox{c}}$, Naoki Genra, Yi-Zhi Huang, Shashank Kanade, Robert McRae, Antun Milas and David Ridout for useful discussions. We also thank the referee for their valuable comments and suggestions. T. C is supported by NSERC $\#$RES0020460.}
\author{Thomas Creutzig}
\address{(T. Creutzig) Department of Mathematical and Statistical Sciences, University of Alberta, 632 CAB, Edmonton, Alberta, Canada T6G 2G1.}
\email{creutzig@ualberta.ca}
\author{Jinwei Yang}
\address{(J. Yang) Department of Mathematical and Statistical Sciences, University of Alberta, 632 CAB, Edmonton, Alberta, Canada T6G 2G1.}
\email{jinwei2@ualberta.ca}

\input xy
\xyoption{all}

\usepackage{amsmath, amssymb, amsthm, amsfonts}
\usepackage{hyperref}

\numberwithin{equation}{section}

\theoremstyle{plain}
\newtheorem{theorem}[subsubsection]{Theorem}
\newtheorem{lemma}[subsubsection]{Lemma}
\newtheorem{prop}[subsubsection]{Proposition}
\newtheorem{cor}[subsubsection]{Corollary}
\newtheorem{conj}[subsubsection]{Conjecture}

\theoremstyle{definition}
\newtheorem{definition}[subsubsection]{Definition}

\newtheorem{remark}[subsubsection]{Remark}
\newtheorem{exam}[subsubsection]{Example}

\newtheorem{conjecture}[subsubsection]{Conjecture}
\def\CC{\mathbb{C}}

\def\LL{\mathbb{L}}

\def\NN{\mathbb{N}}

\def\QQ{\mathbb{Q}}

\def\ZZ{\mathbb{Z}}

\newcommand\cA{\mathcal{A}}
\newcommand\cB{\mathcal{B}}
\newcommand\cC{\mathcal{C}}

\newcommand\cF{\mathcal{F}}
\newcommand\cG{\mathcal{G}}

\newcommand\cM{\mathcal{M}}

\newcommand\cO{\mathcal{O}}

\newcommand\cR{\mathcal{R}}

\newcommand\cU{\mathcal{U}}
\newcommand\cV{\mathcal{V}}
\newcommand\cW{\mathcal{W}}

\newcommand\frg{\mathfrak{g}}
\newcommand\frh{\mathfrak{h}}

\newcommand\id{\textup{id}}
\newcommand{\Ind}{\textup{Ind}}

\newcommand\Span{\textup{Span}}

\newcommand{\tr}{\textup{tr}}

\newcommand{\wt}{\textup{wt}\;}

\newcommand\Hom{\textup{Hom}}

\newcommand{\Ext}{\textup{Ext}}

\newcommand\gl{\mathfrak{gl}}

\renewcommand\sl{\mathfrak{sl}}

\newcommand\so{\mathfrak{so}}

\newcommand{\btimes}{\boxtimes}
\newcommand{\jiao}[1]{\langle{#1}\rangle}

\newcommand\quash[1]{}

\newcommand{\ov}{\overline}

\newcommand\xr{\xrightarrow}

\newcommand\one{\mathbf{1}}

\renewcommand\a\alpha
\renewcommand\b\beta
\newcommand\g\gamma
\renewcommand\d\delta
\newcommand\D\Delta

\renewcommand{\l}{\lambda}

\newcommand{\om}{\omega}

\makeatletter
\newlength{\@pxlwd} \newlength{\@rulewd} \newlength{\@pxlht}
\catcode`.=\active \catcode`B=\active \catcode`:=\active
\catcode`|=\active
\def\sprite#1(#2,#3)[#4,#5]{
   \edef\@sprbox{\expandafter\@cdr\string#1\@nil @box}
   \expandafter\newsavebox\csname\@sprbox\endcsname
   \edef#1{\expandafter\usebox\csname\@sprbox\endcsname}
   \expandafter\setbox\csname\@sprbox\endcsname =\hbox\bgroup
   \vbox\bgroup
  \catcode`.=\active\catcode`B=\active\catcode`:=\active\catcode`|=\active
      \@pxlwd=#4 \divide\@pxlwd by #3 \@rulewd=\@pxlwd
      \@pxlht=#5 \divide\@pxlht by #2
      \def .{\hskip \@pxlwd \ignorespaces}
      \def B{\@ifnextchar B{\advance\@rulewd by \@pxlwd}{\vrule
         height \@pxlht width \@rulewd depth 0 pt \@rulewd=\@pxlwd}}
      \def :{\hbox\bgroup\vrule height \@pxlht width 0pt depth
0pt\ignorespaces}
      \def |{\vrule height \@pxlht width 0pt depth 0pt\egroup
         \prevdepth= -1000 pt}
   }
\def\endsprite{\egroup\egroup}
\catcode`.=12 \catcode`B=11 \catcode`:=12 \catcode`|=12\relax
\makeatother

\def\hboxtr{\FormOfHboxtr} 
\sprite{\FormOfHboxtr}(25,25)[0.5 em, 1.2 ex] 

:BBBBBBBBBBBBBBBBBBBBBBBBB | :BB......................B |
:B.B.....................B | :B..B....................B |
:B...B...................B | :B....B..................B |
:B.....B.................B | :B......B................B |
:B.......B...............B | :B........B..............B |
:B.........B.............B | :B..........B............B |
:B...........B...........B | :B............B..........B |
:B.............B.........B | :B..............B........B |
:B...............B.......B | :B................B......B |
:B.................B.....B | :B..................B....B |
:B...................B...B | :B....................B..B |
:B.....................B.B | :B......................BB |
:BBBBBBBBBBBBBBBBBBBBBBBBB |

\endsprite
\sprite{\FormOfShboxtr}(25,25)[0.3 em, 0.72 ex] 

:BBBBBBBBBBBBBBBBBBBBBBBBB | :BB......................B |
:B.B.....................B | :B..B....................B |
:B...B...................B | :B....B..................B |
:B.....B.................B | :B......B................B |
:B.......B...............B | :B........B..............B |
:B.........B.............B | :B..........B............B |
:B...........B...........B | :B............B..........B |
:B.............B.........B | :B..............B........B |
:B...............B.......B | :B................B......B |
:B.................B.....B | :B..................B....B |
:B...................B...B | :B....................B..B |
:B.....................B.B | :B......................BB |
:BBBBBBBBBBBBBBBBBBBBBBBBB |

\endsprite

\begin{document}

\begin{abstract}
We show that if $V$ is a vertex operator algebra such that all the irreducible ordinary $V$-modules are $C_1$-cofinite and all the grading-restricted generalized Verma modules for $V$ are of finite length, then the category of finite length generalized $V$-modules has a braided tensor category structure.

By applying the general theorem to the simple affine vertex operator algebra (resp. superalgebra) associated to a finite simple Lie algebra (resp. Lie superalgebra) $\frg$ at level $k$ and the category $KL_k(\frg)$ of its finite length generalized modules,
we discover several families of $KL_k(\frg)$ at non-admissible levels $k$, having braided tensor category structures. In particular, $KL_k(\frg)$ has a braided tensor category structure if the category of ordinary modules is semisimple or more generally if the category of ordinary modules is of finite length.

We also prove the rigidity and determine the fusion rules of some categories $KL_k(\frg)$, including the category $KL_{-1}(\sl_n)$. Using these results, we construct a rigid tensor category structure on a full subcategory of $KL_1(\mathfrak{sl}(n|m))$ consisting of objects with semisimple Cartan subalgebra actions.
\end{abstract}

\maketitle
\section{Introduction}
\subsection{Tensor categories of affine Lie algebras}
Tensor category structures on the representation categories of affine Lie algebras, more generally, of vertex operator algebras have been studied since the late 1980s. In the work of physicists (\cite{BPZ, KZ, MS1, MS2}), it has been realized that the category of integrable modules for an affine Lie algebra at a fixed positive integral level should have the structure of a rigid braided tensor category. Kazhdan and Lusztig were the first to rigorously construct a rigid tensor category structure on the category of finite length modules whose composition factors are simple highest weight modules with dominant weights for affine Lie algebras in the case that the level plus the dual Coxeter number is not a positive rational number \cite{KL1}--\cite{KL5}. Then, several works, including those of Beilinson, Feigin and Mazur \cite{BFM} and Huang and Lepowsky \cite{HL7}, gave the constructions of braided tensor categories at positive integral levels; these are in fact shown to be rigid and even modular tensor categories \cite{H4}. Braided tensor category structure on the category $KL_k$ at an admissible level $k$ was recently found by Huang together with the two authors \cite{CHY} and then the rigidity could also be proven in the case that the Lie algebra is simply-laced \cite{C2}. There are however levels that are neither generic nor admissible and in these cases not much is known about the structure of $KL_k$, especially the existence of braided tensor category structure is open. We refer to \cite{H7} for a brief review of vertex-operator-algebraic constructions of various tensor category structures on representation categories for affine Lie algebras. In this work, we partially solved Problem 4.12 there.

\subsection{Main results}
In this paper, we prove that if $V$ is a vertex operator algebra whose irreducible ordinary modules are all {\em $C_1$-cofinite} and whose grading-restricted generalized Verma modules are of finite length, then there is a braided tensor category structure on the category of finite length generalized $V$-modules. Our constructions are based on the tensor category theory of vertex operator algebras developed in \cite{HL1}-\cite{HL7} by Huang, Lepowsky and its logarithmic generalization \cite{HLZ0}-\cite{HLZ8} by Huang, Lepowsky and Zhang. The notion of $C_1$-cofinite modules plays a crucial role in the applicability of this theory. We show that under the above assumption, the category of lower bounded $C_1$-cofinite modules is the same as the category of finite length generalized modules (Theorem \ref{maintheorem}). Using this equivalence, we prove the existence of the braided tensor category structure (Theorem \ref{tensorcategory}). We then apply the general results to affine vertex operator algebras. In particular, if any generalized Verma module for a simple (or any other homomorphic image of the universal) affine vertex operator algebra is of finite length, then the category $KL_k$ of finite length generalized modules has a braided tensor category structure (Theorem \ref{mainresult}).

Our results provide a universal way to construct braided tensor categories for affine Lie algebras at various levels. It not only applies to the module categories studied by Kazhdan and Lusztig when the level plus the dual Coxeter number is non-positive rational and to the categories of ordinary modules at admissible levels (including the positive integral levels), but also gives tensor category structures at some other rational levels that are neither generic nor admissible. In particular, if the category of ordinary modules is semisimple, then $KL_k$ has a braided tensor category structure (Corollary \ref{main}).

\subsection{Applications}
There are three families of not necessarily admissible levels at which the categories $KL_k$ can be shown to have braided tensor category structures:
\begin{enumerate}
\item Collapsing levels for the minimal $W$-algebras $\cW_k(\frg, \theta)$ (Section \ref{sec:collaps});
\item Positive rational but generic levels (Section \ref{sec:generic});
\item The levels $k$ such that the ordinary modules for $\cW_k(\mathfrak{g}, \theta)$ are all of finite length (Section \ref{sec:minimal}).
\end{enumerate}
The first and second families have been studied in \cite{AKMPP1}-\cite{AKMPP2} and \cite{C1}, respectively. In both cases the categories of ordinary modules are semisimple. The third family provides a few interesting examples of not necessarily semisimple categories of ordinary modules via minimal $W$-algebras $\cW_k(\frg, \theta)$ (Theorem \ref{thm:minimal}), particularly when $\cW_k(\frg, \theta)$ is $C_2$-cofinite (Corollary \ref{cor:walgebra}) or when the category of ordinary modules for $\cW_k(\frg, \theta)$ is semisimple (Corollary \ref{semisimplicity}). Examples of these type beyond admissible levels have been obtained in \cite{ArM1, ArM2} and \cite{Ka} (see Example \ref{exam:ArM}).

A further example comes from Theorem \ref{CL}, recently proved by the first author and Linshaw \cite{CL}, that implies that the category of ordinary modules for $\cW_{-2}(\mathfrak{sl}_n, \theta)$, $n\geq 6$ is semisimple and hence $KL_{-2}(\sl_n)$ is a braided tensor category (Theorem \ref{cor:beyondcollapsing}).

As a consequence of our main result, there are braided tensor category structures on the categories $KL_k(\frg)$ at these cases. Here is a list of pairs $(\frg, k)$ beyond admissible levels for which $KL_k(\frg)$ is braided tensor:
\begin{enumerate}
\item $\frg=A_1$, $k=-2+\frac{1}{p}$ and $p\in \mathbb Z_{\geq 1}$.
\item $\frg=A_\ell$, $\ell \geq 1$ and $k = -1$.
\item $\frg=A_\ell$, $\ell \geq 3$ and $k = -2$.
\item $\frg=A_{2\ell-1}$, $\ell \geq 2$ and $k = -\ell$.
\item $\frg=B_\ell$, $\ell \geq 2$ and $k = -2$.
\item $\frg=C_{\ell}$,  and $k = -1-\ell/2$.
\item $\frg=D_\ell$, $\ell \geq 3$ and $k \in \{ -2, -1\}$.
\item $\frg=D_{2\ell-1}$, $\ell \geq 3$ and $k = -2\ell+3$.
\item $\frg=E_6$ and $k \in \{ -4, -3, -2, -1\}$.
\item $\frg=E_7$ and $k \in \{ -6, -4, -3, -2, -1\}$.
\item $\frg=E_8$ and $k \in \{-6, -5, -4, -3, -2, -1\}$.
\item $\frg=F_4$ and $k = -3$.
\item $\frg=G_2$ and $k = -1$.
\end{enumerate}

\subsection{Rigidity}
A major open problem in tensor category theory is rigidity and for this two strategies have been employed. In \cite{C2}, the coset realization of principal $W$-algebras of simply-laced Lie algebras \cite{ACL} was used to prove an equivalence of braided tensor categories between $KL_k$ at admissible level $k$ and certain fusion subcategories of the corresponding principal $W$-algebras. Especially the category $KL_k$ is then rigid as well.

A second strategy is to look for certain abelian intertwining algebras $V$ with compact automorphism group $G$, such that $V^G$ is an affine vertex algebra. Due to McRae \cite{M} (see also \cite{CM} for the case when $G$ is finite and abelian), if a certain category of $V^G$-modules has a braided tensor category structure, then the subcategory generated by the $V^G$-modules showing up in the decomposition of $V$ as a $V^G \times G$-module is a braided tensor category equivalent to an abelian $3$-cocycle modification of ${\rm Rep} \; G$ (we call this modification a {\em twist}). In particular, this subcategory of $V^G$-modules is rigid.

Finding such algebras is tricky and jointly with Adamovi$\acute{\mbox{c}}$ and Genra in \cite{ACGY}, we construct a family of abelian intertwining algebras denoted by $\mathcal V^{(p)}$ for $p \in \mathbb Z_{\geq 1}$, such that $(\mathcal V^{(p)})^{SU(2)} \cong L_{-2+\frac{1}{p}}(\mathfrak{sl}_2)$. It turns out that all simple objects in $KL_{-2+\frac{1}{p}}(\sl_2)$ appear in the decomposition of $\mathcal V^{(p)}$ as an $L_{-2+\frac{1}{p}}(\mathfrak{sl}_2)\times SU(2)$-module. Due to our Proposition \ref{generic}, the assumptions of \cite{M} are satisfied and hence $KL_{-2+\frac{1}{p}}(\sl_2)$ is rigid (Corollary \ref{orbten}, Example \ref{vp}, cf. \cite{ACGY}).
This result completes the construction of rigid braided tensor category structures on $KL_k(\sl_2)$ at all levels $k$ except the critical level $k = -2$.

We employ a similar strategy with $\beta\gamma^{\otimes n}$ and $G=U(1)$ to prove the rigidity of the tensor category $KL_{-1}(\sl_n)$ for the simple affine vertex operator algebra $L_{-1}(\mathfrak{sl}_n)$. We also use the conformal embeddings \cite[Thm.~5.1]{AKMPP4} to study certain subcategories of ordinary modules for $L_{-2}(\mathfrak{sl}_5)$, $L_{-2}(\mathfrak{sl}_6)$, $L_{-3}(\mathfrak{so}_{10})$ and $L_{-4}(\mathfrak{e}_6)$. In all these cases we find rigid vertex tensor categories consisting of simple currents.

Moreover, by applying the tensor category theory for the vertex operator superalgebra extensions to $L_{-1}(\sl_m) \subset L_1(\sl(n|m))$, we are able to determine the category $KL_1^{ss}(\sl(n|m))$ of finite length generalized modules with semisimple Cartan subalgebra actions for the vertex operator superalgebras $L_{1}(\mathfrak{sl}(n|m))$ when $n\neq m$ and also for $L_{1}(\mathfrak{psl}(n|n))$ and prove the rigidity.

\subsection{Outlook}
The main results of this paper hold for certain non-semisimple categories $KL_k$, for example, the category of finite length modules for affine Lie algebras at negative levels studied in \cite{KL1}-\cite{KL5} and the category $KL_k(\frg)$ with $\cW_k(\frg,\theta)$ being $C_2$-cofinite. We are aiming to study more non-semisimple such examples in the future.

One can verify that our results naturally generalize to finite length generalized modules for vertex operator superalgebras, in particular, vertex operator superalgebras associated to Lie superalgebras, the first important case to study is the category of finite length generalized modules for $L_k(\gl(1|1))$ \cite{CR1}. Recently jointly with McRae \cite{CMY2}, we establish the existence of the tensor structures on $KL_k(\gl(1|1))$ and $KL_k^{ss}(\gl(1|1))$ using the main results in this paper, and also study the detailed tensor structures there. We aim to discuss braided tensor categories of finite length generalized modules for a general affine vertex superalgebra in future work. We expect many examples of categories that are not semi-simple but of finite length.

The long-term goal is to establish vertex tensor category results beyond grading-restricted generalized modules, e.g. relaxed-highest weight modules and their spectral flow conjugates of admissible level affine vertex algebras \cite{CR2, KR1, KR2}. This is quite difficult since these categories do not satisfy nice known finiteness conditions as $C_1$-cofiniteness.

Of course one does not only want to establish the existence of tensor category, but also use it for further interesting insights. In the present work, we use the conformal embeddings of \cite{AKMPP1}--\cite{AKMPP4} in order to derive fusion rules and rigidity statements. The tensor category tool is rather useful when it restricts to the branching rules, that is, the decomposition of the larger vertex algebra in terms of modules of the subalgebra (see \cite{CKM2}). However, such decompositions are often quite difficult to determine. Jointly with Adamovi$\acute{\mbox{c}}$ and Genra, we aim to study several open decomposition problems and then use our results for new fusion rules and rigidity statements. This started with the recent work \cite{ACGY}.

In \cite{ACGY}, our results have been used to prove conjectures of \cite{C1} that the three families of vertex algebras: the $\mathcal R^{(p)}$-algebras of \cite{A}, certain $W$-algebras at boundary admissible level and certain chiral algebras coming from physics are all isomorphic. This especially solved the branching rule problem for conformal embeddings of $W$-algebras in $L_k(\mathfrak{sl}_2)$ at certain non-admissible levels $k$.
All these cases are related to $\frg=\mathfrak{sl}_2$, but in \cite{C1}, there are also conjectures for $\frg=\mathfrak{sl}_n$ and higher rank chiral algebras motivated from physics are constructed in \cite{C3}. Again it seems that $W$-algebras and affine vertex algebras at non-admissible and non-generic levels play a central role.

Another conjecture from physics concerning non-admissible level affine vertex algebras is that large coupling limits of certain corner vertex algebras appearing in $S$-duality are large extensions of $L_k(\frg)$ at level $k=-h^\vee+\frac{1}{p}$ for some positive integer $p$ (\cite{CG, CGL}). Furthermore, these vertex algebras are expected to have the compact Lie group $G$ whose Lie algebra is $\frg$ as subgroup of automorphisms. Jointly with Adamovi$\acute{\mbox{c}}$ and Genra in \cite{ACGY}, we proved this conjecture for $G=SU(2)$ and then used it together with the results of this work to determine the rigidity of the tensor category $KL_{-2+\frac{1}{p}}$ as mentioned above. For higher rank, the problem is completely open and progress would be quite exciting. Note also that $S$-duality, associated vertex algebras and their vertex tensor categories are closely connected with the quantum geometric Langlands program (\cite{CG, FG}) and the just mentioned large coupling limits are viewed as classical limits.

\section{Vertex operator algebras}
\subsection{Preliminaries on vertex operator algebras}
Let $(V, Y, {\bf 1}, \omega)$ be a vertex operator algebra. We first recall the definitions of various $V$-modules.
\begin{definition}
\begin{itemize}
\item[(1)]  A \textit{weak $V$-module} (or simply {\em $V$-module}) is a vector space $W$
equipped with a vertex operator map
$$\begin{array}{cccl}
 Y_W: & V & \rightarrow & (\mathrm{End}\,W)[[x,x^{-1}]]\\
 &   v & \mapsto & Y_W(v,x)=\sum_{n\in\mathbb{Z}} v_n\,x^{-n-1}\\
\end{array}$$
satisfying the following axioms:
\begin{itemize}
\item[(i)] The lower truncation condition: For $u\in V$, $w \in W$, $Y_W(u,x)w$ has only finitely many
terms of negative powers in $x$.

 \item[(ii)] The vacuum property:
\begin{equation*}
 Y_W(\mathbf{1},x)=1_W.
 \end{equation*}

 \item[(iii)] The Jacobi identity: for $u,v\in V$,
\begin{align*}
 x_0^{-1}&\delta\left(\dfrac{x_1-x_2}{x_0}\right) Y_W(u,x_1)Y_W(v,x_2) \\
&\quad - x_0^{-1}\delta\left(\dfrac{-x_2+x_1}{x_0}\right)Y_W(v,x_2)Y_W(u,x_1)
\nonumber\\
 & = x_2^{-1}\delta\left(\dfrac{x_1-x_0}{x_2}\right)Y_W(Y(u,x_0)u,x_2).
\end{align*}

\item[(iv)] The Virasoro algebra relations: if we write $Y_W(\omega,x)=\sum_{n\in\mathbb{Z}} L_n x^{-n-2}$,
then for any $m,n\in\mathbb{Z}$,
\begin{equation*}
 [L_m,L_n]=(m-n)L_{m+n}+\dfrac{m^3-m}{12} \delta_{m+n,0} c,
\end{equation*}
where $c$ is the central charge of $V$.

\item[(v)] The $L_{-1}$-derivative property: for any $v\in V$,
\begin{equation*}
 Y_W(L_{-1}v,x)=\dfrac{d}{dx} Y_W(v,x).
 \end{equation*}
\end{itemize}

\item[(2)] A \textit{generalized $V$-module} is a weak $V$-module $(W, Y_{W})$ with a
$\CC$-grading $W=\coprod_{n\in \CC}W_{[n]}$ such that $W_{[n]}$ for $n\in \CC$ are
generalized eigenspaces for the operator $L_{0}$ with eigenvalues $n$.

\item[(3)]
A \textit{lower-bounded generalized $V$-module} is a generalized $V$-module such that for any $n\in\mathbb{C}$, $W_{[n+m]}=0$ for $m\in\mathbb{Z}$ sufficiently negative. A \textit{grading-restricted generalized $V$-module} is a lower-bounded generalized $V$-module such that dim $W_{[n]}<\infty$ for any $n\in\mathbb{C}$.

\item[(4)] An \textit{ordinary $V$-module} is a grading-restricted
generalized $V$-module such that $W_{[n]}$ for $n\in \CC$ are
eigenspaces $W_{(n)}$ for the operator $L_{0}$ (with eigenvalues $n$).

\item[(5)]A generalized $V$-module $W$ is of {\em length $l$} if there exist generalized $V$-submodules
$W = W_1 \supset \cdots \supset W_{l+1} = 0$ such that $W_i/W_{i+1}$ for $i = 1, \dots, l$ are irreducible
ordinary $V$-modules.  A {\em finite length} generalized $V$-module is a generalized $V$-module of length $l$ for some $l \in \ZZ_{\geq 1}$.

\item[(6)]We say a category of $V$-modules is {\em semisimple} if every object is completely reducible.
\end{itemize}
\end{definition}

\begin{remark} The definition of finite length generalized module is due to Huang \cite[Definition~1.2]{H5}. We stress that a finite length module is defined to be a generalized $V$-module whose simple composition factors are ordinary modules.
\end{remark}

\begin{remark}
The following properties of finite length generalized modules are obvious (cf. \cite{H5}):
\begin{itemize}
\item[(1)]A finite length generalized module is grading-restricted.
\item[(2)]The category of finite length generalized $V$-modules is
closed under the operation of direct sum, taking generalized $V$-submodules,
quotient generalized $V$-submodules and contragredient duals.
\end{itemize}
\end{remark}

Let $V$ be a vertex operator algebra and let $(W, Y_W)$ be a lower bounded generalized $V$-module with
\[
W = \coprod_{n \in \mathbb{C}}W_{[n]}
\]
where for $n \in \mathbb{C}$, $W_{[n]}$ is the generalized weight space with weight $n$. Its {\em contragredient module} is the vector space
\[
W' = \coprod_{n \in \mathbb{C}}(W_{[n]})^*,
\]
equipped with the vertex operator map $Y'$ defined by
\[
\langle Y'(v,x)w', w \rangle = \langle w', Y^{\circ}_W(v,x)w \rangle
\]
for any $v \in V$, $w' \in W'$ and $w \in W$, where
\[
Y^{\circ}_W(v,x) = Y_W(e^{xL(1)}(-x^{-2})^{L_0}v, x^{-1}),
\]
for any $v \in V$, is the {\em opposite vertex operator} (cf. \cite{FHL}).
We also use the standard notation
\[
\overline{W} = \prod_{n \in \mathbb{C}}W_{[n]},
\]
for the formal completion of $W$ with respect to the $\mathbb{C}$-grading.

We also need the notion of logarithmic intertwining operators:
\begin{definition}\label{log:def}
Let $(W_1,Y_1)$, $(W_2,Y_2)$ and $(W_3,Y_3)$ be $V$-modules. A {\em logarithmic intertwining
operator of type ${W_3\choose W_1\,W_2}$} is a linear map
\begin{align*}
\mathcal{Y}: &W_1\otimes W_2\to W_3\{x\}[\log x] \\
&w_{(1)}\otimes w_{(2)}\mapsto{\mathcal{Y}}(w_{(1)},x)w_{(2)}=\sum_{k=0}^{K}\sum_{n\in
{\mathbb C}}{w_{(1)}}_{n, k}^{
\mathcal{Y}}w_{(2)}x^{-n-1}(\log x)^{k}
\end{align*}
for all $w_{(1)}\in W_1$ and $w_{(2)}\in W_2$, such that the
following conditions are satisfied:
\begin{itemize}
\item[(i)] The {\em lower truncation
condition}: for any $w_{(1)}\in W_1$, $w_{(2)}\in W_2$, $n\in
{\mathbb C}$ and $k=0, \dots, K$,
\begin{equation*}
{w_{(1)}}_{n+m, k}^{\mathcal{Y}}w_{(2)}=0\;\;\mbox{ for }\;m\in {\mathbb
N} \;\mbox{sufficiently large}.
\end{equation*}
\item[(ii)] The {\em Jacobi identity}:
\begin{eqnarray*}
\lefteqn{ x^{-1}_0\delta \bigg( {x_1-x_2\over x_0}\bigg)
Y_3(v,x_1){\mathcal{Y}}(w_{(1)},x_2)w_{(2)}} \\
&&\hspace{2em}- x^{-1}_0\delta \bigg( {x_2-x_1\over -x_0}\bigg)
{\mathcal{Y}}(w_{(1)},x_2)Y_2(v,x_1)w_{(2)} \\
&&{ = x^{-1}_2\delta \bigg( {x_1-x_0\over x_2}\bigg){
\mathcal{Y}}(Y_1(v,x_0)w_{(1)},x_2) w_{(2)}}
\end{eqnarray*}
for $v\in V$, $w_{(1)}\in W_1$ and $w_{(2)}\in W_2$.
\item[(iii)] The {\em $L_{-1}$-derivative property:} for any
$w_{(1)}\in W_1$,
\begin{equation*}
{\mathcal{Y}}(L_{-1}w_{(1)},x)=\frac d{dx}{\mathcal{Y}}(w_{(1)},x).
\end{equation*}
\end{itemize}
A logarithmic intertwining operator is called an {\it intertwining operator} if no $\log x$ appears in ${\mathcal{Y}}(w_{(1)},x)w_{(2)}$
for $w_{(1)}\in W_1$ and $w_{(2)}\in W_2$. The dimension of the space  $\mathcal{V}_{W_1 W_2}^{W_3}$ of all logarithmic intertwining
operators is called the {\it fusion rule}.
\end{definition}

The following finiteness condition plays an important role in this paper:
\begin{definition}
Let $V$ be a vertex algebra and let $W$ be a weak $V$-module. Let $C_1(W)$ be the subspace of $W$ spanned by elements of the form $u_{-1}w$ where $u\in V_+:=\coprod_{n\in \mathbb{Z}_{>0}} V_{(n)}$ and $w\in W$. We say that $W$ is {\it $C_1$-cofinite} if $W/C_1(W)$ is finite dimensional.
\end{definition}

The $C_1$-cofinite condition is closed under taking quotients:
\begin{lemma}[\cite{H5}] \label{qu}
Let $V$ be a vertex algebra. For a $V$-module $W_1$ and a $V$-submodule $W_2$ of $W_1$, $C_1(W_1/W_2)=(C_1(W_1)+W_2)/W_2$.	
\end{lemma}	

We shall also need the following fact:
\begin{lemma}[\cite{H5}]\label{inclusion}
If $W$ is a finite length module with $C_1$-cofinite simple composition factors, then $W$ is also $C_1$-cofinite.
\end{lemma}

\subsection{Generalized Verma modules}\label{GVM}
This subsection is mainly due to \cite{Li} (cf. \cite{DLM2}). Let $(V, Y, \one, \om)$ be a vertex operator algebra. Let $W$ be a lower bounded generalized $V$-module and let
\[
W^{top} = \Span_{\CC}\{w \in W|w \;\mbox{homogeneous}, \; w' = 0\; \mbox{if}\; \Re(\wt w') < \Re(\wt w)\}.
\]
It is easy to see that $W^{top}$ is a module for the Zhu's algebra $A(V)$ and
\[
W^{top} \cap C_1(W) = 0.
\]

Consider the quotient space
\[
\widehat{V} = V\otimes \CC[t, t^{-1}]/D(V\otimes \CC[t, t^{-1}]),
\]
where $D = L(-1)\otimes 1 + 1\otimes \frac{d}{dt}$. Denote by $v(m)$ the image of $v\otimes t^m$ in $\widehat{V}$ for $v \in V$ and $m \in \ZZ$. Then $\widehat{V}$ is $\ZZ$-graded by defining the degree of $v(m)$ to be $\wt v -m-1$ for homogeneous $v$. Denote the homogeneous space of degree $m$ by $\widehat{V}(m)$. The space $\widehat{V}$ is in fact a $\ZZ$-graded Lie algebra with Lie bracket
\[
[u(m), v(n)] = \sum_{i \in \NN}\binom{m}{i}u_iv(m+n-i).
\]
In particular, $\widehat{V}(0)$ is a Lie subalgebra. Let $U$ be a $\widehat{V}(0)$-module, it can be lifted to a module for $P = \bigoplus_{i \geq 0}\widehat{V}(i)$ by letting $\widehat{V}(i)$, $i > 0$ act trivially. Set
\[
M(U) = \cU(\widehat{V})\otimes_{\cU(P)}U.
\]
Let $J(U)$ be the intersection of all $\ker \a$ where $\a$ runs over all $\widehat{V}$-homomorphisms from $M(U)$ to weak $V$-modules. Set
\[
F(U) = M(U)/J(U).
\]
Then $F(U)$ is a weak $V$-module, called {\em generalized Verma module}. From the construction, $F(U)$ has the following universal property: for any weak $V$-module $M$ such that $M^{top} = U$, there is a unique weak $V$-module homomorphism $\tilde{\phi}: F(U)\rightarrow M$. Note that $F(U)$ might be zero.

There is a natural Lie algebra homomorphism $\widehat{V}(0)\rightarrow A(V)$ that sends $v(\wt v -1)$ to the image of $v \in V$ in $A(V)$, thus for any $A(V)$-module $U$, it can be lifted to $\widehat{V}(0)$-module. One can construct $F(U)$ as follows: Define for $v \in V$,
\[
Y_{M(U)}(v, z) = \sum_{m \in \ZZ}v(m)z^{-m-1},
\]
and let $I$ be the subspace of $M(U)$ spanned by the coefficients of
\[
(z_0+z_2)^{\wt u}Y_{M(U)}(u, z_0 + z_2)Y_{M(U)}(v, z_2)w = (z_2+z_0)^{\wt u}Y_{M(U)}(Y(u,z_0)v, z_2)w
\]
for any homogeneous $u, v \in V$ and $w \in U$. Then the space
\[
F(U) = M(U)/\cU(\widehat{V})I
\]
becomes a weak $V$-module such that $F(U)^{top} = U$. In particular, $F(U)$ is not zero.

Take $w \in W^{top}$, let $\jiao{w}$ be the $A(V)$-module generated by $w$ and let $V\cdot w$, called the highest weight module, be the $V$-submodule of $W$ generated by $w$ (See \cite[Sec.~4.5]{LL}). By the universal property, there is a surjective $V$-module homomorphism $F(\jiao{w}) \twoheadrightarrow V \cdot w$.

If $W$ is $C_1$-cofinite, then $W^{top}$ is finite dimensional. Therefore $F(\jiao{w})$ is a grading-restricted generalized $V$-module.

\subsection{Affine vertex operator algebras}
Vertex operator algebras associated to affine Lie algebras are the most important examples of vertex operator algebras. Let us briefly recall the construction of affine vertex operator algebras. The affine Lie algebra $\widehat{\mathfrak{g}}$ associated with a finite simple Lie algebra $\mathfrak{g}$ and $(\cdot, \cdot)$ is the vector space $\mathfrak{g} \otimes \CC[t, t^{-1}] \oplus \CC{\bf k}$ equipped with the bracket operation defined by
\[
[a \otimes t^m, b \otimes t^n] = [a, b]\otimes t^{m+n} + (a, b)m\delta_{m+n,0}{\bf k},
\]
\[
[a \otimes t^m, {\bf k}] = 0,
\]
for $a, b \in \mathfrak{g}$ and $m, n \in \ZZ$. It is $\ZZ$-graded in a natural way. Consider the subalgebras
\[
\widehat{\mathfrak{g}}_{\pm} = \mathfrak{g}\otimes t^{\pm 1}\CC[t^{\pm 1}]
\]
and the vector space decomposition
\[
\widehat{\mathfrak{g}} = \widehat{\mathfrak{g}}_{-} \oplus \mathfrak{g} \oplus \CC{\bf k} \oplus \widehat{\mathfrak{g}}_{+}.
\]

For $\lambda \in \mathfrak{h}^{*}$, the dual space of a Cartan subalgebra $\frh \subset \frg$, we use $E^\lambda$ to denote the irreducible highest weight $\mathfrak{g}$-module with highest weight $\lambda$.

Let $M$ be a $\mathfrak{g}$-module, viewed as homogeneously graded of fixed degree, and let $k \in \CC$. Let $\widehat{\mathfrak{g}}_{+}$ act on $M$ trivially and ${\bf k}$ act as the scalar multiplication by $k$. Then $M$ becomes a $\mathfrak{g} \oplus \CC{\bf k} \oplus \widehat{\mathfrak{g}}_{+}$-module, and we have the $\CC$-graded induced $\widehat{\mathfrak{g}}$-module
\[
\widehat{M}_{k} = U(\widehat{\mathfrak{g}})\otimes_{U(\mathfrak{g}\oplus \CC{\bf k}\oplus \hat{\mathfrak{g}}_{+})}M,
\]
which contains a canonical copy of $M$ as its subspace.

For $\lambda \in \mathfrak{h}^{*}$, we use the notation $V_k(\lambda)$, called the {\em generalized Verma modules}, to denote the graded $\widehat{\mathfrak{g}}$-module $\widehat{E^\lambda}_{k}$. Let $J_k(\lambda)$ be the maximal proper submodule of $V_k(\lambda)$ and $L_k(\lambda) = V_k(\lambda)/J_k(\lambda)$. Then $L_k(\lambda)$ is the unique irreducible graded $\widehat{\mathfrak{g}}$-module such that ${\bf k}$ acts as $k$ and the space of all elements annihilated by $\widehat{\mathfrak{g}}_{+}$ is isomorphic to the $\mathfrak{g}$-module $E^\lambda$. It is clear that the generalized Verma module $V_k(\lambda)$ is grading-restricted if and only if $\l \in P_+$, the set of dominant integral weights of $\frg$.

\begin{remark}
By the universal property, the generalized Verma module $V_k(\lambda)$ is the same as the generalized Verma module $F(E^{\lambda})$ constructed in Subsection \ref{GVM}.
\end{remark}

For $k \neq -h^{\vee}$, the quotient modules of $V_k(0)$ have vertex operator algebra structure (\cite{FZ}, \cite{LL}), we denote the vertex operator algebras $V_k(0)$ and $L_k(0)$ by $V_k(\mathfrak{g})$ and $L_k(\mathfrak{g})$, respectively.

In the special case that $\mathfrak{g}$ is a one-dimensional abelian Lie algebra with a basis $\{\alpha\}$, let $M_{\alpha}(k)$ (or simply $M(k)$) be the universal Heisenberg vertex operator algebra of level $k$ generated by $\alpha$. The irreducible $M_{\alpha}(k)$-modules are the modules $M_{\alpha}(k, \lambda)$ (or simply $M(k,\lambda)$) generated by a vector $v_{\lambda}$ with the action given by $\alpha(n)v_{\lambda} = \delta_{n,0}\lambda v_{\lambda}$ for $n \in \ZZ_{\geq 0}$.


\subsection{Tensor category of vertex operator superalgebra extensions}
In this subsection, we will recall tensor category theory of vertex operator superalgebra extensions developed mainly in \cite{HKL, KO, CKM1, CKM2}.

Let $V$ be a vertex operator algebra and $(\cC, \btimes, \one_{\cC}, \cA_{\bullet,\bullet,\bullet}, l_{\bullet}, r_{\bullet}, \cR_{\bullet,\bullet})$ be a category of $V$-modules with a natural vertex and braided tensor category structure in the sense of \cite{HLZ0}-\cite{HLZ8}, which we will briefly summarize in Section 3. Note that unit object $\one_{\cC}$ is actually the vertex operator algebra $V$ itself. A triple $(A, \mu_A, \iota_A)$ with $A$ an object in $\cC$ and $\mu_A: A\btimes A \rightarrow A$, $\iota_A: \one_{\cC} \rightarrow A$ morphisms in $\cC$ is called an {\em associative algebra} if:
\begin{itemize}
\item[(1)]Multiplication is associative: $\mu_A \circ (\id_A \btimes \mu_A) = \mu_A \circ (\mu_A \btimes \id_A)\circ \cA_{A,A,A}: A\btimes (A \btimes A) \rightarrow A$
\item[(2)]Multiplication is unital: $\mu_A \circ (\iota_A \btimes \id_A) = l_A: \one_{\cC}\btimes A \rightarrow A$ and $\mu_A \circ (\id_A \btimes \iota_A) = r_A: A \btimes \one_{\cC} \rightarrow A$
\end{itemize}
We say $(A, \mu_A, \iota_A)$ is a {\em commutative algebra} if additionally
\begin{itemize}
\item[(3)]Multiplication is commutative: $\mu_A \circ \cR_{A,A} = \mu_A: A \btimes A \rightarrow A$.
\end{itemize}

For an associative algebra $(A, \mu_A, \iota_A)$, define $\cC_A$ to be the category of pairs $(X, \mu_X)$, where $X$ is an object of $\cC$ and $\mu_X: A\btimes X \rightarrow X$ is a morphism in $\cC$ satisfy the following:
\begin{itemize}
\item[(1)]Unit property: $l_X = \mu_X \circ (\iota_A \btimes \id_X): \one_{\cC}\btimes X \rightarrow X$
\item[(2)]Associativity: $\mu_X \circ (\id_A \btimes \mu_X) = \mu_X \circ (\mu_A \btimes \id_X)\circ \cA_{A, A, X}: A \btimes (A \btimes X) \rightarrow X$.
\end{itemize}
A morphism $f \in \Hom_{\cC_A}((X_1, \mu_{X_1}), (X_2, \mu_{X_2}))$ is a morphism $f \in \Hom_{\cC}(X_1, X_2)$ such that $\mu_{X_2}\circ (\id_A \btimes f) = f\circ \mu_{X_1}$.

When $A$ is commutative, define $\cC_A^{loc}$ to be the full subcategory of $\cC_A$ containing {\em local} (or {\em dyslectic}) objects: those $(X, \mu_X)$ such that $\mu_X \circ \cR_{X,A} \circ \cR_{A, X} = \mu_X$.

Let $W$ be an object of $\cC$, define $\cF_A(W) = A \btimes W$ and
\[
\mu_{\cF_A(W)}: A \btimes (A \btimes W) \xr{\cA_{A,A,W}}  (A \btimes A) \btimes W \xr{\mu_A\btimes \id_W} A \btimes W.
\]
Then $(\cF_A(W), \mu_{\cF_A(W)})$ is an object of $\cC_A$. The {\em induction functor}, which we also denote by $\cF_A$, is given by
\begin{align*}
\cF_A: \cC &\rightarrow \cC_A\\
W &\mapsto (\cF_A(W), \mu_{\cF_A(W)});\;\;\; f \mapsto \id_A \btimes f.
\end{align*}
The {\em restriction functor} $\cG_A: \cC_A \rightarrow \cC$ is given by
\[
\cG_A: (X, \mu_X) \mapsto X; \;\;\; f \mapsto f.
\]

\begin{theorem}\label{sum1}
Let $\cC$ be a braided tensor category and let $A$ be a commutative and associative algebra in $\cC$. Then the following results hold:
\begin{itemize}
\item[(1)] The category $\cC_A$ of $A$-modules is a tensor category (\cite[Thm.~1.5]{KO}, \cite[Thm.~2.53]{CKM1}), and the subcategory $\cC_A^{loc}$ of local $A$-modules is a braided tensor category (\cite[Thm.~1.10]{KO}, \cite[Thm.~2.55]{CKM1}).
\item[(2)]The induction functor $\cF_A: \cC \rightarrow \cC_A$ is tensor functor (\cite[Thm.~1.6]{KO}, \cite[Thm.~2.59]{CKM1}).
\item[(3)]The induction functor satisfies Frobenius reciprocity, that is, it is left adjoint to the restriction functor $\cG_A$ from $\cC_A$ to $\cC$:
\begin{equation}\label{eqn:fro-rec}
\Hom_{\cC_A}(\cF_A(W), X) = \Hom_{\cC}(W, \cG_A(X))
\end{equation}
for objects $W$ in $\cC$ and $X$ in $\cC_A$ (see for example \cite[Thm.~1.6(2)]{KO}, \cite[Lem.~7.8.12]{EGNO}, \cite[Lem.~2.61]{CKM1}).
\item[(4)]Let $W$ be an object in $\cC$. Then $\cF_A(W)$ is in $\cC_A^{loc}$ if and only if the monodromy isomorphism $\cM_{A, W} : = \cR_{W,A}\circ \cR_{A,W} = \id_{A\btimes W}$ (\cite[Prop.~2.65]{CKM1}).
\item[(5)]Let $\cC^0$ be the full subcategory of objects in $\cC$ that induce to $\cC_A^{loc}$. Then the restriction of the induction functor $\cF_A: \cC^0 \rightarrow \cC_A^{loc}$ is a braided tensor functor (\cite[Thm.~2.67]{CKM1}).
\end{itemize}
\end{theorem}

The following theorem allows one to study the tensor structure of the category of modules in $\cC$ for a vertex operator (super)algebra extension of $V$ using abstract tensor category theory:
\begin{theorem}\label{sum2}
Let $V$ be a vertex operator algebra, and let $\cC$ be a category of $V$-modules with a natural vertex tensor category structure. Then the following results hold:
\begin{itemize}
\item[(1)] A vertex operator algebra extension $V \subset A$ in $\cC$ is equivalent to a commutative associative algebra in the braided tensor category $\cC$ with trivial twist and injective unit (\cite[Thm.~3.2]{HKL}).
\item[(2)] The category of modules in $\cC$ for the extended vertex operator algebra $A$ is isomorphic to the category of local $\cC$-algebra modules $\cC_A^{loc}$ (\cite[Thm.~3.4]{HKL}).
\item[(3)] The results in (1) and (2) hold for a vertex operator superalgebra extension: The vertex operator superalgebra extension $V \subset A$ in $\cC$ such that $V$ is in the even subalgebra $A^0$ is equivalent to a commutative associative superalgebra in $\cC$ whose twist $\theta$ satisfies $\theta^2 = \id_A$. The category of generalized modules for the vertex operator superalgebra $A$ is isomorphic to the category of local $\cC$-superalgebra modules $\cC_A^{loc}$ (\cite[Thm.~3.13, 3.14]{CKL}).
\item[(4)] The isomorphism given in \cite[Thm.~3.4]{HKL} and \cite[Thm.~3.14]{CKL} between the category of modules in $\cC$ for the extended vertex operator (super)algebra $A$ and the category of local $\cC$-algebra modules $\cC_A^{loc}$ is an isomorphism of vertex tensor (super)categories (\cite[Thm.~3.65]{CKM1}).
\end{itemize}
\end{theorem}

The following proposition gives an criterion for the simplicity of the induced objects:
\begin{prop}\cite[Prop.~4.4]{CKM1}\label{prop:simplicity}
Let $A$ be a vertex operator (super)algebra extension of $V$ with $V \subset A^0$, the even subalgebra of $A$
and both $A$, $V$ simple, and suppose $A = \oplus_{i \in I}A_i$ as a $V$-module with each $A_i$ a simple $V$-module. If $W$ is a simple $V$-module such that each non-zero $A_i\btimes_V W$ is a simple $V$-module and each non-zero $A_i\btimes_V W \ncong A_j \btimes_V W$ unless $i=j$, then $\cF_A(W)$ is a simple object of $\cC_A^0$.
\end{prop}

If $A$ is an infinite direct sum of objects in the category $\cC$, thus it is not an object of $\cC$ itself, we have to consider {\em direct sum completion} $\cC^{\oplus}$ (\cite{AR, CKM2}), or more generally {\em direct limit completion} $\Ind(\cC)$ of the category $\cC$ (\cite{CMY1}). In this paper, we shall only need the direct sum completion $\cC^{\oplus}$ whose objects are direct sums $\bigoplus_{s \in S}X_s$ of objects in $\cC$, where $S$ is an arbitrary index set, and whose morphisms $f: \bigoplus_{s \in S}X_s \rightarrow \bigoplus_{t \in T}Y_t$ are such that for $s \in S$, $f|_{X_s}$ maps to $\bigoplus_{t \in T'}Y_t$ for some finite subset $T' \subset T$ (see a detailed definition in \cite{AR}, cf. \cite[Appendix~A]{CKM2}).

If $\cC$ is a braided tensor category, then direct sum completion $\cC^{\oplus}$ inherits the same tensor structure from $\cC$ by defining all the structure isomorphisms componentwise (\cite{AR}). Moreover there is a braided monoidal functor $\cC \rightarrow \cC^{\oplus}$ which is fully faithful, that is, bijective on morphisms. If $\cC$ is semisimple, then $\cC^{\oplus}$ is also an abelian category. As a result, for a semisimple braided tensor category $\cC$, all the results in Theorem \ref{sum1}, Theorem \ref{sum2} and Proposition \ref{prop:simplicity} naturally apply with $\cC$ replaced by its direct sum completion $\cC^{\oplus}$.

\section{Tensor categories of vertex operator algebras}
The logarithmic tensor category theory of Huang, Lepowsky and Zhang (\cite{HLZ0}--\cite{HLZ8}) gave a construction of a braided tensor category structure on an appropriate module category $\cC$ for a vertex operator algebra. In this section, we will give a sufficient condition for the existence of this tensor category structure by verifying the assumptions needed for invoking the theory, i.e., \cite[Assump.~10.1]{HLZ6} and \cite[Assump.~12.1, 12.2]{HLZ8}.

The \cite[Assump.~10.1]{HLZ6} mainly requires that the category $\cC$ is closed under taking submodules, quotients, contragredient duals and {\em $P(z)$-tensor products}, and \cite[Assump.~12.2]{HLZ8}, for verifying the {\em associativity axioms}, can be replaced by the conditions in Theorem \ref{assiso} below (see \cite[Thm.~3.1]{H6}). We shall briefly recall the construction of the $P(z)$-tensor products and the associativity axiom in order to check the applicability of the logarithmic tensor category theory.

\subsection{Construction of the $P(z)$-tensor product}
In the tensor category theory for vertex operator algebras, the tensor product bifunctors
are not built on the classical tensor product bifunctor for vector spaces. Instead, the central concept underlying the constructions is the notion of $P(z)$-tensor product (\cite{HL3}, \cite{HLZ3}, \cite{HLZ4}), where $z$ is a nonzero complex number and $P(z)$ is the Riemann sphere $\overline{\CC}$ with one negative oriented puncture at $\infty$ and two ordered positively oriented punctures at $z$ and $0$, with local coordinates $1/w$, $w-z$ and $w$, respectively.

We briefly recall the construction of $P(z)$-tensor product in \cite{HLZ4}. Let $W_1, W_2 \in \cC$, let $v \in V$ and
\[
Y_t(v, x) = \sum_{n \in Z}(v \otimes t^n)x^{-n-1} \in (V \otimes \mathbb{C}[t, t^{-1}])[[x, x^{-1}]].
\]
Denote by $\tau_{P(z)}$ the action of
\[
V \otimes \iota_{+}{\mathbb C}[t,t^{- 1}, (z^{-1}-t)^{-1}]
\]
on the vector space $(W_1 \otimes W_2)^*$, where $\iota_+$ is the operation of expanding a rational function in the formal variable $t$ in the direction of positive powers of $t$, given by
\begin{equation}\label{defofpz}
\begin{split}
 \bigg(\tau_{P(z)}\bigg(x_0^{-1}\delta\bigg(\frac{x_1^{-1}-z}{x_0}\bigg)&Y_t(v, x_1)\bigg)\lambda\bigg)(w_{(1)}\otimes w_{(2)}) = \\
& \;\;\;\; z^{-1}\delta\bigg(\frac{x_1^{-1}-x_0}{z}\bigg)\lambda(Y_1(e^{x_1L(1)}(-x_1^{-2})^{L(0)}v, x_0)w_{(1)}\otimes w_{(2)}) \\
&\quad  +\; x_0^{-1}\delta\bigg(\frac{z-x_1^{-1}}{-x_0}\bigg)\lambda(w_{(1)}\otimes Y_2^{\circ}(v, x_1)w_{(2)}),
\end{split}
\end{equation}
for $v \in V$, $\lambda \in (W_1 \otimes W_2)^*$, $w_{(1)} \in W_1$ and $w_{(2)} \in W_2$. Denote by $Y_{P(z)}'$ the action of $V \otimes \mathbb{C}[t,t^{-1}]$ on $(W_1\otimes W_2)^*$ defined by
\[
Y_{P(z)}'(v,x) = \tau_{P(z)}(Y_t(v,x)).
\]
We also have the operators $L_{P(z)}'(n)$ for $n \in \mathbb{Z}$ defined by
\begin{equation}
Y_{P(z)}'(\omega,x) = \sum_{n \in \ZZ}L_{P(z)}'(n)x^{-n-2}.
\end{equation}

Given two objects $W_1$ and $W_2$ in $\cC$, let $W_1\hboxtr_{P(z)}W_2$ be the vector space consisting of all the elements $\lambda \in (W_1 \otimes W_2)^*$ satisfying the following two conditions:
\begin{itemize}
\item[(1)]{\it $P(z)$-compatibility condition}:
\begin{itemize}
\item[(a)]{\it Lower truncation condition}: For all $v \in V$, the formal Laurent series $Y_{P(z)}'(v,x)\lambda$ involves only finitely many negative powers of $x$.
\item[(b)] The following formula holds for all $v$ in $V$:
\begin{equation}\nonumber
\tau_{P(z)}\bigg(z^{-1}\delta\bigg(\frac{x_1-x_0}{z}\bigg)Y_t(v,x_0)\bigg)\lambda
  = z^{-1}\delta\bigg(\frac{x_1-x_0}{z}\bigg)Y_{P(z)}'(v,x_0)\lambda .
\end{equation}
\end{itemize}
\item[(2)] {\it $P(z)$-local grading restriction condition}:
\begin{itemize}
\item[(a)] {\it Grading condition}: $\lambda$ is a (finite) sum of generalized eigenvectors of $(W_1 \otimes W_2)^*$ for the operator $L_{P(z)}'(0)$.
\item[(b)]The smallest subspace $W_{\lambda}$ of $(W_1 \otimes W_2)^*$ containing $\lambda$ and stable under the component operators $\tau_{P(z)}(v\otimes t^n)$ of the operators $Y_{P(z)}'(v,x)$ for $v \in V$, $n \in \mathbb{Z}$, have the properties:
    \begin{equation*}
    \mbox{dim}\; (W_{\lambda})_{[n]} < \infty
    \end{equation*}
    \[
    (W_{\lambda})_{[n+k]} = 0\;\;\;\mbox{for $k \in \mathbb{Z}$ sufficiently negative};
    \]
    for any $n \in \mathbb{C}$, where the subscripts denote the $\mathbb{C}$-grading by $L_{P(z)}'(0)$-eigenvalues.
\end{itemize}
\end{itemize}

The following theorem gives the construction of $P(z)$-tensor product on an appropriate category $\mathcal{C}$ of $V$-modules:
\begin{theorem}[\cite{HLZ4}]
The vector space $W_1\hboxtr_{P(z)}W_2$ is closed under the action $Y_{P(z)}'$ of $V$ and the Jacobi identity holds on $W_1\hboxtr_{P(z)}W_2$. Furthermore, if $\cC$ is closed under taking contragredients, then the $P(z)$-tensor product of $W_1, W_2 \in \mathcal{C}$ exists if and only if $W_1\hboxtr_{P(z)}W_2$ equipped with $Y_{P(z)}'$ is an object of $\mathcal{C}$ and in this case, the $P(z)$-tensor product is the contragredient module of $(W_1\hboxtr_{P(z)}W_2, Y_{P(z)}')$.
\end{theorem}

\subsection{Associativity isomorphism}
The associativity isomorphism of the $P(z)$-tensor product is the most important ingredient of Huang-Lepowsky-Zhang's tensor category theory. The following theorem, which is essentially a combination of \cite[Thm.~10.3]{HLZ6}, \cite[Thm.~11.4]{HLZ7}, and \cite[Thm.~3.1]{H6}, gives the sufficient condition for the associativity isomorphism:
\begin{theorem}\label{assiso}
Assume that $V$ satisfies the following conditions:
\begin{enumerate}
\item [(1)] For any two modules $W_1$ and $W_2$ in $\mathcal{C}$ and any $z \in \CC^{\times}$, if the generalized $V$-module $W_{\lambda}$ generated by a generalized eigenvector $\lambda \in (W_1 \otimes W_2)^{*}$ for $L_{P(z)}'(0)$ satisfying the $P(z)$-compatibility condition is lower bounded, then $W_{\lambda}$ is an object of $\mathcal{C}$.
\item[(2)]The product or the iterates of the intertwining operators for $V$ have the convergence and extension property (see definitions in \cite[Sec.~11.1]{HLZ7}).
\end{enumerate}
Then for any $V$-module $W_1, W_2$ and $W_3$ and any complex number $z_1$ and $z_2$ satisfying $|z_1|>|z_2|>|z_1-z_2|>0$, there is a unique isomorphism $\mathcal{A}_{P(z_1), P(z_2)}^{P(z_1-z_2), P(z_2)}$ from $W_1 \boxtimes_{P(z_1)} (W_2 \boxtimes_{P(z_2)} W_3)$ to $(W_1 \boxtimes_{P(z_1-z_2)} W_2) \boxtimes_{P(z_2)} W_3$ such that for $w_{(1)} \in W_1$, $w_{(2)} \in W_2$, $w_{(3)} \in W_3$,
\begin{eqnarray*}
&&\overline{\mathcal{A}}_{P(z_1), P(z_2)}^{P(z_1-z_2), P(z_2)}(w_1 \boxtimes_{P(z_1)} (w_2 \boxtimes_{P(z_2)} w_3))\nonumber\\
&& \;\;\; = (w_1 \boxtimes_{P(z_1-z_2)} w_2) \boxtimes_{P(z_2)} w_3,
\end{eqnarray*}
where
\[
\overline{\mathcal{A}}_{P(z_1), P(z_2)}^{P(z_1-z_2), P(z_2)}: \overline{W_1 \boxtimes_{P(z_1)} (W_2 \boxtimes_{P(z_2)} W_3)} \rightarrow \overline{(W_1 \boxtimes_{P(z_1-z_2)} W_2) \boxtimes_{P(z_2)} W_3}
\]
is the canonical extension of $\mathcal{A}_{P(z_1), P(z_2)}^{P(z_1-z_2), P(z_2)}$.
\end{theorem}
	
\subsection{Applicability of tensor category theory}
We first introduce the following conjecture (see also discussions in \cite[Sec.~6.3]{CMR}):
\begin{conjecture}\label{mainconjecture}
Let $V$ be a vertex operator algebra such that all the irreducible ordinary $V$-modules are $C_1$-cofinite. Then we conjecture that $W$ is a lower bounded $C_1$-cofinite module if and only if $W$ is a finite length generalized module.
\end{conjecture}
\begin{remark}
One direction of Conjecture \ref{mainconjecture} is obvious, that is, under the assumption that all the irreducible ordinary $V$-modules are $C_1$-cofinite, if $W$ is of finite length, $W$ is $C_1$-cofinite by Lemma \ref{inclusion}.
\end{remark}

\begin{remark}
Let $V$ be a homomorphic image $\widetilde{V_k(\frg)}$ of the universal affine vertex operator algebra. Then all its ordinary irreducible modules are $C_1$-cofinite.
\end{remark}

Let $\cO^{fin}$ be the category of finite length generalized modules. Assume Conjecture \ref{mainconjecture} holds, i.e., $\cO^{fin}$ is the same as the category of lower bounded $C_1$-cofinite modules, we can establish the braided tensor category structure on $\cO^{fin}$ using Huang-Lepowsky-Zhang's logarithmic tensor category theory:
\begin{theorem}\label{tensorcategory}
Suppose Conjecture \ref{mainconjecture} holds, then $\cO^{fin}$ has a braided tensor category structure with the tensor product bifunctor $\boxtimes_{P(1)}$, the unit object $V$ and the braiding isomorphism, the associativity isomorphism, the left and right unit isomorphisms given in \cite[Sec.~12.2]{HLZ8}.
\end{theorem}
\begin{proof}
The proof is a natural generalization of those of \cite[Cor.~4.1.7, Thm.~4.2.5]{CJORY} for the Virasoro algebra case.

It is clear that the category $\cO^{fin}$ is closed under taking finite direct sums, submodules, quotients and contragredient duals. Take $W_1, W_2 \in \cO^{fin}$, then by Conjecture \ref{mainconjecture}, $W_1, W_2$ are grading restricted $C_1$-cofinite modules. In \cite{Mi}, Miyamoto proved that $W_1\boxtimes_{P(z)} W_2$ is lower bounded $C_1$-cofinite, thus $W_1\boxtimes_{P(z)} W_2 \in \cO^{fin}$. The category $\cO^{fin}$ is closed under $P(z)$-tensor product.

Since all the objects in $\cO^{fin}$ are grading-restricted $C_1$-cofinite modules, from \cite{H2}, matrix elements of products and iterates of logarithmic intertwining operators satisfy certain differential equations with regular singular singularities. By \cite{HLZ7}, the convergence and extension properties needed in the associativity isomorphism holds for $\cO^{fin}$ (see also Theorem \ref{assiso} (2)).

It remains to show Condition (1) in Theorem \ref{assiso} holds for $\cO^{fin}$. Let $W_{\lambda}$ be the lower bounded generalized $V$-module constructed in Theorem \ref{assiso}. The inclusion map $W_{\lambda} \subset (W_1\otimes W_2)^*$ intertwines the action of $V \otimes \iota_{+}{\mathbb C}[t,t^{- 1}, (z^{-1}-t)^{-1}]$ given by $\tau_{P(z)}$, therefore by \cite[Prop.~5.24]{HLZ4}, it corresponds to a $P(z)$-intertwining map $I$ of type ${W_{\lambda}' \choose W_1\,W_2}$ such that for an arbitrary element $w' \in W_{\l}$
\[
w'(w_1\otimes w_2)=\langle w', I(w_1\otimes w_2)\rangle.
\]
By \cite[Prop.~4.8]{HLZ3}, there is an intertwining
operator $\mathcal{Y}$ of type ${W_{\lambda}' \choose W_1\,W_2}$ such that
\[
\mathcal{Y}(w_1, z)w_2 = I(w_1\otimes w_2).
\]
This intertwining operator is surjective since
\[
\langle w', \mathcal{Y}(w_1, z)w_2\rangle = \langle w', I(w_1\otimes w_2)\rangle = w'(w_1\otimes w_2).
\]
By \cite{Mi} (cf. \cite[Cor.~2.14]{CMY1}), $W_{\lambda}'$ is $C_1$-cofinite since $W_1$ and $W_2$ are $C_1$-cofinite. Then it follows from Conjecture \ref{mainconjecture} that $W_{\lambda}' \in \cO^{fin}$, thus $W_{\lambda} \in \cO^{fin}$. We verified condition (1) in Theorem \ref{assiso}. Therefore, the associativity isomorphisms hold for $\cO^{fin}$.

Now that \cite[Assump.~10.1]{HLZ6} and associativity axiom hold, it follows from \cite[Thm.~12.15]{HLZ8} that $\cO^{fin}$ is braided tensor.
\end{proof}


In the following theorem, we give a sufficient condition for an arbitrary vertex operator algebra in order for Conjecture \ref{mainconjecture} to be true:

\begin{theorem}\label{maintheorem}
Let $V$ be a vertex operator algebra such that all the irreducible ordinary $V$-modules are $C_1$-cofinite. If all the grading-restricted generalized Verma modules for $V$ are of finite length, then Conjecture \ref{mainconjecture} holds. Moreover, the category $\cO^{fin}$ has a braided tensor category structure.
\end{theorem}
\begin{proof} If $W$ is of finite length, then from Lemma \ref{inclusion} $W$ is $C_1$-cofinite.

Assume $W$ is a lower bounded $C_1$-cofinite $V$-module. Since $W$ is $C_1$-cofinite, $W^{top}$ is a finite dimensional $\widehat{V}(0)$-module. Take an irreducible $\widehat{V}(0)$-submodule in $W^{top}$, it generates a highest weight $V$-submodule $W_1$ of $W$. Apparently $W_1 \nsubseteq C_1(W)$ and $W_1$ is a quotient of a certain grading-restricted generalized Verma module.

From Lemma \ref{qu}
$$(W/W_1)/C_1(W/W_1)\cong W/(C_1(W)+W_1),$$
it is clear that $W/W_1$ is also a lower bounded $C_1$-cofinite module.
Following the previous procedure, we can pick up an irreducible $\widehat{V}(0)$-module in the top level space of $W/W_1$, and denote the highest weight module generated by this vector in $W/W_1$ by $\overline{W_2}$, it follows that there exists $W_2\subset W$ such that $W_1\subset W_2$ and $W_2/W_1\cong \overline{W_2}$.
Thus we have a sequence of epimorphisms
\begin{align*}
W\twoheadrightarrow W/W_1\twoheadrightarrow (W/W_1)/(W_2/W_1)\cong W/ W_2
\end{align*}
Continuing in this manner we obtain a sequence of epimorphisms
\begin{align*}
W\twoheadrightarrow W/W_1\twoheadrightarrow W/ W_2 \twoheadrightarrow W/W_3 \twoheadrightarrow \cdots \twoheadrightarrow W/W_{n-1}\twoheadrightarrow W/W_{n}.
\end{align*}
Note that for $i = 1, \dots, n-1$
\[
\dim (W/W_{i+1})/C_1(W/W_{i+1}) < \dim (W/W_{i})/C_1(W/W_{i}),
\]
the process ends in a finite number of steps since  dim $W/C_1(W)$ is finite. Moreover, the last step is such that $(W/W_{n})/C_1(W/W_{n})=0$ which implies that $W/W_{n}=0$ (see \cite[Lem.~3.1.1]{CJORY}).
Thus, we have obtained a sequence of grading-restricted generalized modules
\begin{equation}\label{highestweightfiltration}
0\hookrightarrow W_1 \hookrightarrow W_2\cdots \hookrightarrow W_{n-1}\hookrightarrow W_{n} = W
\end{equation}
such that $W_{i}/W_{i-1}\cong \overline{W_{i}}$ is a highest weight $V$-module.
Since $\ov{W_i}$ are all quotients of certain grading-restricted generalized Verma modules, they are of finite length by assumption. By refining the filtration (\ref{highestweightfiltration}) we obtain a finite composition series for $W$. Therefore $W \in \cO^{fin}$.
\end{proof}

\subsection{Tensor categories of affine Lie algebras}
We are interested in the following categories of modules for the simple affine vertex operator algebra $L_k(\frg)$:
\begin{itemize}
\item $\cO_k(\frg)$ (or simply $\cO_k$ if $\mathfrak{g}$ is clear): the category of grading-restricted generalized modules.
\item $KL_k(\mathfrak{g})$ (or simply $KL_k$): the category of finite length generalized modules.
\item $\cO_{k,{\rm ord}}(\frg)$ (or simply $\cO_{k,{\rm ord}}$): the category of ordinary modules.
\end{itemize}
It is clear that the categories $KL_k$ and $\cO_{k,{\rm ord}}$ are both full subcategories of the category $\cO_k$.
It is shown in \cite{Ar1} that if $k$ is admissible, then these categories are the same. In this paper, we are mainly interested in the levels $k$ that are neither generic nor admissible. We first have the following consequence of Theorem \ref{maintheorem}:

\begin{cor}\label{characterization}
If all the grading-restricted generalized Verma modules are of finite length, then an object $W$ is in $KL_k$ if and only if $W$ is a lower bounded $C_1$-cofinite module.
\end{cor}

\begin{remark}
The assumption in Corollary \ref{characterization} is related to the famous Kazhdan-Lusztig polynomials. For $\lambda \in P_{+}$, V. Deodhar, O. Gabbar and V. Kac (\cite{DGK}) proved that the highest weight of any irreducible subquotient of the Verma module with highest weight $\lambda$ has the form $w \circ \lambda$ for some $w \in W$, where $w\circ \lambda = w(\lambda + \rho) - \rho$ and $W$ is the affine Weyl group (see also \cite{KK}). The multiplicities are given by the Kazhdan-Lusztig polynomials, which were conjectured in \cite{DGK} and proved by M. Kashiwara \cite{Kas}, Kashiwara-Tanisaki \cite{KasT} and by Casian \cite{Ca}.
\end{remark}

As a consequence of Theorem \ref{tensorcategory} and Corollary \ref{characterization}, we have:
\begin{theorem}\label{mainresult}
If all the grading-restricted generalized Verma modules are of finite length, then $KL_k$ has a braided tensor category structure.
\end{theorem}

\begin{remark}
We have derived the results for $KL_k$, the category of finite length generalized modules for the simple affine vertex operator algebra $L_k(\frg)$, but all the results in this section also hold for the other affine vertex operator algebras that are quotients of the universal affine vertex operator algebras $V_k(\frg)$.
\end{remark}

Theorem \ref{mainresult} applies to all the known results on the tensor category structure of affine Lie algebras so far, in particular, it gives another proof of the tensor category structure constructed in \cite{KL1}-\cite{KL5} (see also \cite{Zh})
\begin{cor}
Let $k + h^{\vee} \notin \QQ_{\geq 0}$, then the category of finite length $L_k(\mathfrak{g})$-modules with simple composition factors isomorphic to $L_k(\lambda)$ for $\lambda \in P_+$ has a braided tensor category structure.
\end{cor}
\begin{proof} By \cite[Prop.~2.14]{KL1}, the generalized Verma module $V_k(\lambda)$ for $\lambda \in P_+$ is of finite length. Then the conclusion follows from Theorem \ref{mainresult}.
\end{proof}

Another important application of Theorem \ref{mainresult} is as follows:
\begin{cor}\label{main}
Assume the category $\cO_{k,{\rm ord}}$ of ordinary $L_k(\frg)$-modules is semisimple, then the category $KL_k$ has a natural structure of braided tensor category.
\end{cor}

The category $\cO_{k,{\rm ord}}$ is semisimple when $k\in \ZZ_{\geq 0}$, $k+h^{\vee} \notin \QQ$ and $k$ admissible (\cite{FZ}, \cite{KL1}, \cite{Ar1}). Therefore Corollary \ref{main} gives another proof of the existence of the tensor category structure on $KL_k$ for these $k$. In the following section, we will investigate more such examples.

\section{Tensor category of \texorpdfstring{$KL_k$}{}}
In the previous section, we have proved that if all the grading-restricted generalized Verma modules are of finite length, then there is a braided tensor category structure on the category $KL_k$. In particular, if the category $\cO_{k,{\rm ord}}$ is semisimple, then $KL_k$ is a braided tensor category. In this section, we will provide several families of affine Lie algebras at certain levels such that the assumption above holds, and hence $KL_k$ has a tensor category structure. Our technique is mainly via minimal $W$-algebras, which we will briefly introduce here.

\subsection{Minimal $W$-algebras}\label{sec:minimal}
A universal $W$-algebra is a vertex algebra $\mathcal{W}^k(\mathfrak{g}, f)$ associated to the triple $(\mathfrak{g}, f, k)$, where $\mathfrak{g}$ is a basic simple Lie superalgebra with a non-degenerate invariant supersymmetric bilinear form $(\cdot, \cdot)$, $f$ is a nilpotent element of $\mathfrak{g}_{\overline{0}}$, and $k \in \CC$, by applying the quantum Hamiltonian reduction functor $H_f$ to a complex associated to the affine vertex superalgebra $V_k(\mathfrak{g})$ \cite{KRW, KW}. In particular, it was shown that, for $k$ non-critical and $f$ part of an $\mathfrak{sl}_2$-triple, $\mathcal{W}^k(\mathfrak{g}, f)$ has a conformal vertex algebra structure with a conformal vector $\omega$. For $k$ non-critical the vertex algebra $\mathcal{W}^k(\mathfrak{g}, f)$ has a unique simple quotient, denoted by $\mathcal{W}_k(\mathfrak{g}, f)$.

Choose a Cartan subalgebra $\mathfrak{h}$ for $\mathfrak{g}_{\overline{0}}$ and let $\Delta$ be the set of roots. Fix a minimal root $-\theta$ of $\mathfrak{g}$. We may choose root vectors $e_{\theta}$ and $e_{-\theta}$ such that
\[
[e_{\theta}, e_{-\theta}] = x \in \mathfrak{h},\;\;\; [x, e_{\pm \theta}] = \pm e_{\pm \theta}.
\]
The eigenspace decomposition of ad $x$ gives a minimal $\frac{1}{2}\ZZ$-grading:
\[
\mathfrak{g} = \mathfrak{g}_{-1}\oplus \mathfrak{g}_{-1/2} \oplus \mathfrak{g}_{0} \oplus \mathfrak{g}_{1/2}\oplus \mathfrak{g}_{1},
\]
where $\mathfrak{g}_{\pm 1} = \CC e_{\pm \theta}$. Furthermore, one has
\[
\mathfrak{g}_0 = \mathfrak{g}^{\natural} \oplus \CC x,
\]
where
\[
\mathfrak{g}^{\natural} = \{a \in \mathfrak{g}_0|(a,x) = 0\}.
\]

It is proved in \cite{KW} that the universal minimal $W$-algebras of level $k$ corresponding to $f = e_{-\theta}$ for a minimal root $-\theta$, denoted by
\[
\mathcal{W}^k(\mathfrak{g}, \theta) = \mathcal{W}^k(\mathfrak{g}, e_{-\theta}),
\]
has a unique simple quotient, denoted by
\[
\mathcal{W}_k(\mathfrak{g}, \theta) = \mathcal{W}_k(\mathfrak{g}, e_{-\theta}),
\]
and that $\mathcal{W}^k(\mathfrak{g}, \theta)$ is freely generated by the elements $J^{\{a\}}$ ($a$ runs over a basis of $\mathfrak{g}^{\natural}$), $G^{\{u\}}$ ($u$ runs over a basis of $\mathfrak{g}_{-1/2}$), and the conformal element $\omega$. Let $\mathcal{V}^k(\mathfrak{g}^{\natural})$ be the subalgebra of the vertex operator algebra $\mathcal{W}^k(\mathfrak{g}, \theta)$ generated by the elements $\{J^{\{a\}}|a \in \mathfrak{g}^{\natural}\}$. It is isomorphic to a universal affine vertex algebra. Set $\mathcal{V}_k(\mathfrak{g}^{\natural})$ to be the image of $\mathcal{V}^k(\mathfrak{g}^{\natural})$ in $\mathcal{W}_k(\mathfrak{g}, \theta)$.

Let $H_{\theta}$ be the Hamiltonian reduction functor from the category $\mathcal{O}^k$ consisting of $\widehat{\frg}$-modules with semisimple Cartan subalgebra actions and finite dimensional weight spaces, ${\bf k}$ acting as $k\cdot \id$, and with finitely many maximal weights, to the category of $\mathcal{W}^k(\mathfrak{g}, \theta)$-modules. We need the following important properties of $H_{\theta}$ proven in \cite{Ar3}:
\begin{theorem}\label{hreduction}
We have
\begin{enumerate}
\item $H_{\theta}$ is exact.
\item $H_{\theta}$ sends ordinary $L_k(\mathfrak{g})$-modules to ordinary $\mathcal{W}^k(\mathfrak{g}, \theta)$-modules.
\item If $L(\l)$ is an irreducible highest weight $\widehat{\frg}$-module with highest weight $\l$, then $\jiao{\l, \a_0^{\vee}} \in \ZZ_{\geq 0}$ implies $H_{\theta}(L(\l)) = 0$. Otherwise, $H_{\theta}(L(\l))$ is an irreducible highest weight $\mathcal{W}^k(\mathfrak{g}, \theta)$-module with the corresponding highest weight.
\end{enumerate}
\end{theorem}

Now we prove the following theorem:
\begin{theorem}\label{thm:minimal}
Let $k \in \CC\backslash \ZZ_{\geq 0}$, if every ordinary $\cW_k(\frg, \theta)$-module is of finite length, then every module in $\cO_{k,{\rm ord}}$ is of finite length. Moreover, the category $KL_k$ has a structure of braided tensor category.
\end{theorem}
\begin{proof}
Let $M$ be an object in $\cO_{k,{\rm ord}}$. To prove $M$ is of finite length, it is equivalent to show that any decreasing sequence in $\cO_{k,{\rm ord}}$
\[
M = M_0 \supset M_1 \supset M_2 \supset \cdots \supset M_n \supset \cdots
\]
and any increasing sequence
\[
M_0 \subset M_1 \subset \cdots \subset M_n \subset \cdots \subset M
\]
are stationary, i.e., $M_n = M_{n+1}$ if $n \gg 0$ (\cite[Theorem~2.1]{S}, cf. \cite[Exercise~8.20]{KS}). Applying $H_{\theta}(\cdot)$ to the decreasing sequence, because ordinary $\cW_k(\frg, \theta)$-modules are of finite length, $H_{\theta}(M_n) = H_{\theta}(M_{n+1})$ if $n \gg 0$. Because $H_{\theta}(\cdot)$ is exact, $H_{\theta}(M_n/M_{n+1}) = 0$. If $M_n/M_{n+1} \neq 0$, $M_n/M_{n+1}$ has an irreducible subquotient $L_k(\l)$ for some $\l$ (see for example \cite[Prop.~3.1]{DGK}). Because $k \notin \ZZ_{\geq 0}$, $\jiao{\l, \a_0^{\vee}} \notin \ZZ_{\geq 0}$ (\cite[Remark~3.5]{AKMPP2}), then by Theorem \ref{hreduction} (3), $H_{\theta}(L_k(\l)) \neq 0$ and hence $H_{\theta}(M_n/M_{n+1}) \neq 0$, that is a contradiction. We proved $M_n = M_{n+1}$ if $n \gg 0$. Similarly we can show the increasing sequence is stationary.
\end{proof}

An immediate application of Theorem \ref{thm:minimal} is the following:
\begin{cor}\label{cor:walgebra}
If the minimal $W$-algebra $\cW_{k}(\frg, \theta)$ is $C_2$-cofinite, then $KL_k(\frg)$ has a braided tensor category structure.
\end{cor}
\begin{proof}
In light of Theorem \ref{thm:minimal}, we only need to show every ordinary $\cW_{k}(\frg, \theta)$-module is of finite length. This is guaranteed by \cite[Prop.~3.15]{H5}.
\end{proof}

\begin{exam}\label{exam:ArM}
Besides $W$-algebras coming from admissible affine Lie algebras, there are more families of $C_2$-cofinite $\cW_k(\frg, \theta)$ discovered by Arakawa and Moreau in \cite{ArM1,ArM2}.
\begin{itemize}
\item $\mathfrak{g} = D_4, E_6, E_7, E_8$, $k \geq -\frac{h^{\vee}}{6}-1$ and $k \in \ZZ$;
\item $\frg = D_l$ with $l \geq 5$, $k \geq -2$ and $k \in \ZZ$.
\item $\frg = G_2$ with $k=-1$.
\end{itemize}
By Corollary \ref{cor:walgebra}, the categories $KL_k(\frg)$ for $\frg$ and $k$ listed above have braided tensor category structures.
\end{exam}

The following result is an immediate consequence of \cite[Thm.~5.7]{AKMPP2} and Corollary \ref{main}:
\begin{cor}\label{semisimplicity}
Assume that $\mathfrak{g}$ is a simple Lie algebra and $k \in \CC\backslash \ZZ_{\geq 0}$ such that the category of ordinary $\cW_k(\mathfrak{g}, \theta)$-modules is semisimple. Then the category $\cO_{k,{\rm ord}}$ is semisimple and therefore $KL_k$ has a braided tensor category structure.
\end{cor}
\begin{proof} Assume $M_k(\lambda)$ is a highest weight ordinary $L_k(\mathfrak{g})$-module for $k \in \CC\backslash \ZZ_{\geq 0}$. By Theorem \ref{hreduction} (3), $H_{\theta}(M_k(\lambda))$ is a nonzero module for $\cW_k(\mathfrak{g}, \theta)$. Also because the functor $H_{\theta}(\cdot)$ sends Verma module to Verma module (\cite[Thm.~6.3]{KW}), $H_{\theta}(M_k(\lambda))$ is a highest weight module. Thus it is an irreducible module due to the semisimplicity of the category of ordinary $\cW_k(\mathfrak{g}, \theta)$-modules. Then by \cite[Thm.~5.5, Lem.~5.6]{AKMPP2}, the category of ordinary $L_k(\mathfrak{g})$-modules is semisimple. \end{proof}

\begin{remark}
An important family of rational and $C_2$-cofinite simple $W$-algebras with non-admissible levels is given by Kawasetsu in \cite{Ka}, where the author showed that the vertex operator algebra $\mathcal{W}_k(\mathfrak{g}, \theta)$ is rational and $C_2$-cofinite in the following non-admissible cases:
\begin{itemize}
\item $\mathfrak{g} = D_4, E_6, E_7, E_8$ and $k = -\frac{h^{\vee}}{6}$.
\end{itemize}
Then by Corollary \ref{semisimplicity} (also by Corollary \ref{cor:walgebra}), the category $KL_k(\mathfrak{g})$ has a braided tensor category structure.
\end{remark}

\subsection{\bf Collapsing levels}\label{sec:collaps}
In this subsection, we will construct tensor category structure on $KL_k$ at certain levels beyond admissible levels, in particular, at {\em collapsing levels}. Recall that $k$ is a collapsing level if $\mathcal{W}_k(\mathfrak{g}, \theta) = \mathcal{V}_k(\mathfrak{g}^{\natural})$. D. Adamovi$\acute{\mbox{c}}$, V. Kac, P. M$\ddot{\mbox{o}}$seneder Frajria, P. Papi and O. Perse (\cite{AKMPP2}) proved that the category of ordinary $L_k(\mathfrak{g})$-modules is semisimple when $k$ is a collapsing level. The set of Lie algebras and their corresponding (non-admissible) collapsing levels is listed as follows:
\begin{theorem}[\cite{AKMPP2}]\label{levels}
The category $\cO_{k,{\rm ord}}(\frg)$ is semisimple in the following cases:
\begin{itemize}
\item[(1)]$\mathfrak{g} = D_{\ell}$, $\ell \geq 3$ and $k = -2$;
\item[(2)]$\mathfrak{g} = B_{\ell}$, $\ell \geq 2$ and $k = -2$;
\item[(3)]$\mathfrak{g} = A_{\ell}$, $\ell \geq 2$ and $k = -1$;
\item[(4)]$\mathfrak{g} = A_{2\ell-1}$, $\ell \geq 2$ and $k = -\ell$;
\item[(5)]$\mathfrak{g} = D_{2\ell-1}$, $\ell \geq 3$ and $k = -2\ell + 3$;
\item[(6)]$\mathfrak{g} = C_{\ell}$, $k = -1-\ell/2$;
\item[(7)]$\mathfrak{g} = E_{6}$, $k = -4$;
\item[(8)]$\mathfrak{g} = E_{7}$, $k = -6$;
\item[(9)]$\mathfrak{g} = F_{4}$, $k = -3$.
\end{itemize}
\end{theorem}

As a consequence of Corollary \ref{main}, we have
\begin{cor}\label{collapsinglevels}
The categories $KL_k(\frg)$ for $\frg$ and $k$ listed in Theorem \ref{levels} have braided tensor category structures.
\end{cor}

For later use, we need the semisimplicity of the category $KL_k(\frg)$ at collapsing levels. In the following result, we prove the category $\cO_k(\frg)$ of grading-restricted generalized modules is semisimple, so is its full subcategory $KL_k$:
\begin{theorem}\label{KLcoinside}
For $k \neq -h^{\vee}$, if all the grading-restricted generalized Verma modules for $L_k(\frg)$ are irreducible, then the category $\cO_k$ of grading-restricted generalized modules is semisimple. In particular, if $k$ is a collapsing level, the category $\cO_k$ is semisimple. Consequently, $\cO_k = \cO_{k,{\rm ord}}$, and $KL_k$ is a full subcategory of $\cO_k$ consisting of finite direct sums of simple objects.
\end{theorem}
\begin{proof}
We prove the claim using \cite[Lem.~1.3.1]{GK}. It suffices to verify (i)-(iv) there. It is obvious that the category $\cO_k$ is closed under taking subquotients and $L_k(\l)$, $\l \in P_+$ exhaust all the simple objects in $\cO_k$. (i) and (iii) are verified.

Let $W$ be an object in $\cO_k$. Since $W^{top}$ is a finite dimensional $\frg$-module, by the universal property of generalized Verma modules, there is a nonzero $L_k(\frg)$-module map from $V_k(\l)$ to a submodule of $W$ for some $\l \in P_+$. Because $V_k(\l)$ is irreducible, it is contained in $W$. We verified (iv).

By \cite[Sec.~3.7 (d)]{KL1},
\[
\Ext^1_{\cO_k}(L_k(\l), L_k(\mu)) = \Ext^1_{\cO_k}(V_k(\l), V_k(\mu)) = \Ext^1_{\cO_k}(V_k(\l), V_k(\ov{\mu})') = 0 \;\;\; \mbox{for} \; \l,\mu \in P_+,
\]
where $\ov{\mu} \in P_+$ satisfying that $E^{\ov{\mu}} = (E^{\mu})'$. Thus (ii) is verified. Consequently $\cO_k$ is semisimple.

It was shown in \cite[Thm.~5.9]{AKMPP2} that the highest weight modules in $\cO_{k,{\rm ord}}(\frg)$ for $(\frg,k)$ listed in Theorem \ref{levels} are irreducible. In particular, the generalized Verma modules are all irreducible. Therefore the conclusion holds for the collapsing levels $k$.
\end{proof}

\begin{remark}
Theorem \ref{KLcoinside} generalizes \cite[Thm.~5.5]{AKMPP2}, which says that if all the highest weight modules for $L_k(\frg)$ in the category $\cO_{k,\mathrm{ord}}$ of ordinary modules are irreducible, then $\cO_{k,\mathrm{ord}}$ is semisimple. But for our purpose (see for example Sec. 6.2), we need the category $\cO_k$, especially its full subcategory $KL_k$ to be semisimple.
\end{remark}

\subsection{\bf Beyond collapsing levels}
Let $n\geq 4$, $r \geq 0$, take $\frg = \sl(n+r+2|r)$, then $\frg^{\natural} = \gl(n+r|r)$ (\cite[Prop.~2.1]{KRW}). Set $A = \cW_{-2}(\mathfrak{sl}(n+r+2|r), \theta)$, and $B = L_{-1}(\gl(n+r|r)) = L_{-1}(\mathfrak{sl}(n+r|r))\otimes \Pi$, where $\Pi$ is a rank one Heisenberg vertex operator algebra, there is a natural embedding $B \hookrightarrow A$ (see for example \cite[Tab.~1--4]{AKMPP1} for the details). Define
\[
C_{n,r} = {\rm Com}(B,A) \overset{\tiny\mbox{def}}{=} \{a \in A|[Y(a,x_1), Y(b,x_2)] = 0 \; \mbox{for all}\; b \in B\}.
\]
We expect that $A$ is a certain simple current extension. For this, we define
\begin{definition}
Let $\cC$ be a braided tensor category and $\mathcal S$ a set of simple currents, that is, invertible objects in $\cC$. Then an object $M$ is called a {\em fixed-point} with respect to $\mathcal S$ if there exist two simple currents $J, J'$ in $\mathcal S$ with $J \not\cong J'$ such that $J \boxtimes M \cong J' \boxtimes M$.
Let $R = \bigoplus_{i \in I} J^i$ be a simple current extension, that is, a commutative superalgebra object, simple as a $R$-module and is a direct sum of inequivalent simple currents. We say $R$ is {\em fixed-point free} in $\cC$ if there is no simple object in $\cC$ that is a fixed-point with respect to $\{J^i\}_{i \in I}$.
\end{definition}
\begin{conj}\label{conjecture}
For $n\geq 4$ the following vertex operator algebras coincide
\[
C_{n,r} \cong \cW_{\ell}(\mathfrak{sl}_{n-2}, f_{prin}),
\]
where $\ell = -h^{\vee} + \frac{n-1}{n}$ and $f_{prin}$ denotes a principle nilpotent element of $\mathfrak{sl}_{n-2}$.
Moreover $A$ is a fixed-point free simple current extension in the category of ordinary modules of $L_{-1}(\mathfrak{sl}(n+r|r))\otimes \Pi \otimes   \cW_{\ell}(\mathfrak{sl}_{n-2}, f_{prin})$.
\end{conj}

Recently, the case $r = 0$ of Conjecture \ref{conjecture} was proved in \cite[Thm.~9.8]{CL}:
\begin{theorem}[\cite{CL}]\label{CL}
For $n \geq 4$, $C_{n,0} \cong \cW_{\ell}(\mathfrak{sl}_{n-2}, f_{prin})$, where $\ell = -h^{\vee} + \frac{n-1}{n}$ and $f_{prin}$ is a principle nilpotent element of $\mathfrak{sl}_{n-2}$. The branching rules are the following
\[
\cW_{-2}(\mathfrak{sl}_{n+2}, \theta) \cong \bigoplus_{s \in \ZZ}L_{-1}(\lambda_s) \otimes M(n(n-2)/(n+2),s) \otimes \LL_{\ell}(\om_{\bar{s}}),
\]
where $\LL_{\ell}(\om_{\bar{s}})$ are irreducible $\cW_{\ell}(\mathfrak{sl}_{n-2}, f_{prin})$-modules (see \cite{C2} for the details), $\bar{s}$ is the residue of $s$ modulo $n-2$.
\end{theorem}

The statement of Theorem \ref{CL} also holds for $n = 3$ and was actually proved in \cite{AKMPP3}: If $\mathfrak{g} = \sl_5$, then $\mathcal{V}_{-2}(\mathfrak{g}^{\natural}) = L_{-1}(\mathfrak{sl}_3)\otimes M_{J^{\{c\}}}(3/5)$, where $\{c\}$ is a basis for $\mathfrak{g}_0^{\natural} = \CC$ (see \cite[Lem.~6.1]{AKMPP3} for the choice of $c$). One has the following isomorphism as $\mathcal{V}_{-2}(\mathfrak{g}^{\natural})$-modules:
\begin{prop}\cite[Corollary 8.3]{AKMPP3}\label{prop:AKMPP3}
\[
\cW_{-2}(\mathfrak{sl}_{5}, \theta) \cong \bigoplus_{s \in \ZZ}L_{-1}(\lambda_s) \otimes M_{J^{\{c\}}}(3/5,s),
\]
where $\lambda_s \overset{\tiny\mbox{def}}{=} s \omega_1$ when $s \geq 0$ and $\lambda_s \overset{\tiny\mbox{def}}{=} -s\omega_2$ when $s < 0$ and $\om_1, \om_2$ are the fundamental weights for $\sl_3$.
\end{prop}

We illustrate Theorem \ref{CL} for $n = 4$:
\begin{cor}\label{thm:conj}
We have the following isomorphism as $L_{-1}(\mathfrak{sl}_4)\otimes M_{J^{\{c\}}}(4/3) \otimes L^{Vir}(1/2, 0)$-modules:
\[
\cW_{-2}(\mathfrak{sl}_{6}, \theta) \cong \bigoplus_{s \in \ZZ}L_{-1}(\lambda_s) \otimes M_{J^{\{c\}}}(4/3,s) \otimes L^{Vir}(1/2, \overline{s}/2),
\]
where $\lambda_s \overset{\tiny\mbox{def}}{=} s\omega_1$, $\lambda_{-s} \overset{\tiny\mbox{def}}{=} s\omega_{3}$ for $s \in \ZZ_{\geq 0}$, and $\overline{s}$ is the residue of $s$ modulo $2$.
\end{cor}
\begin{proof}
Let $\cC$ be the Deligne product of the category $KL_{-1}(\mathfrak{sl}_4)$, the category of ordinary $M_{J^{\{c\}}}(4/3)$-modules with semisimple $J^{\{c\}}(0)$ actions and the category of ordinary Virasoro vertex operator algebra $L^{Vir}(1/2, 0)$-modules. It is clear that $\cC$ is a semisimple braided tensor category, and its direct limit completion $\cC^{\oplus}$ also inherits a braided tensor structure.

Set
\[
J = L_{-1}(\lambda_1) \otimes M_{J^{\{c\}}}(4/3,1) \otimes L^{Vir}(1/2, 1/2)
\]
and
\[
\mathcal{A}_{-2} = \bigoplus_{s \in \ZZ}J^s \;\;\;\;\; \mathrm{in}\;\; \cC^{\oplus}.
\]
Then $\cA_{-2}$ is a simple current extension of $L_{-1}(\mathfrak{sl}_4)\otimes M_{J^{\{c\}}}(4/3) \otimes L^{Vir}(1/2, 0)$ by the fusion rules of irreducible $M(4/3)$ and $L^{Vir}(1/2, 0)$-modules as well as those of irreducible $L_{-1}(\sl_4)$-modules, which will be proved in \eqref{fusionrules}.
It is easy to see that the conformal weight of the top space of $J^s := J^{\btimes s}$ is
\[
h_{J^s} = \frac{s^2+|s|+\overline{s}}{2},
\]
and therefore the twist $\theta_{J^s} := e^{2\pi i h_{J^s}}$ satisfies
\[
\theta_{J^{s}} \in \pm 1,\;\;\; \theta_{J^{s+2}} = \theta_{J^s}.
\]
Moreover, it follows from Lemma \ref{dimJ} that the categorical dimension, i.e. the quantum dimension
$$\dim J^s : = \tr_{\cC} (\id_{J^s}) = (-1)^s.$$
By \cite[Thm.~1.3]{CKL} together with Remark \ref{sce}, $\mathcal{A}_{-2}$ has a $\frac{1}{2}\ZZ$-graded vertex algebra structure. The conformal weight $3/2$ space of $\mathcal{A}_{-2}$ is
\[
(E^{\lambda_1}\otimes \CC v_1 \otimes \CC v_{1/2, 1/2}) \oplus (E^{\lambda_{-1}}\otimes \CC v_{-1} \otimes \CC v_{1/2, 1/2})
\]
and together with the conformal element, weight $1$ elements, they generate the vertex algebra $\mathcal{A}_{-2}$. The conclusion then follows from the uniqueness theorem \cite[Thm.~3.2]{ArCKL}.
\end{proof}
\begin{remark}
The same argument applies replacing $L_{-1}(\mathfrak{sl}_4)$ by $L_{1}(\mathfrak{sl}(r|r+4))\cong L_{-1}(\mathfrak{sl}(r+4|r)) $ and $L_{-1}(\lambda_s)$ by $L_{1}(\lambda_s)$. The top level of the simple current $J$ is then the standard representation of $\mathfrak{sl}(r|r+4)$  and we have to choose the parity of it such that the even part is $r+4$ dimensional and the odd one $r$ dimensional. The resulting extension is then of the same type as the minimal W-superalgebra of type $\mathfrak{sl}(r|r+6)$ and hence by the uniqueness theorem \cite[Thm.~3.2]{ArCKL} it must be $\cW_{2}(\mathfrak{sl}(r|r+6), \theta) \cong \cW_{-2}(\mathfrak{sl}(r+6|r), \theta)$.
\end{remark}

Using Theorem \ref{CL}, we will prove the main result of this subsection:
\begin{theorem}\label{cor:beyondcollapsing}
For $n\geq 5$, the category $KL_{-2}(\sl_n)$ is semisimple and has a braided tensor category structure.
\end{theorem}
\begin{proof}
We first prove the case for $n \geq 6$. In light of \cite[Lem.~5.6]{AKMPP2} (cf. Corollary \ref{semisimplicity}): if the highest weight module $H_{\theta}(U)$ for every nonzero highest weight module $U$ from the category $KL_{-2}(\sl_n)$ is an irreducible $\cW_{-2}(\sl_n, \theta)$-module, then every highest weight module in $KL_{-2}(\sl_n)$ is irreducible. Consequently by Theorem \ref{KLcoinside} and by \cite[Thm.~5.6]{AKMPP2}, the category $KL_{-2}(\sl_n)$ and the category $\cO_{-2, \mathrm{ord}}(\mathfrak{sl}_{n})$ are semisimple. Then it follows from Corollary \ref{main} that $KL_{-2}(\sl_n)$ has a braided tensor category structure.

So it remains to prove that $H_{\theta}(U)$ for every highest weight module $U$ in $KL_{-2}(\sl_n)$ is irreducible, the arguments are essentially the same as the proof of \cite[Prop.~2.7]{DMZ}. Since $H_{\theta}(U)$ is a highest weight module, especially an ordinary module, then $H_{\theta}(U)$ regarded as an $L_{-1}(\mathfrak{sl}_{n-2})\otimes \Pi \otimes \cW_{\ell}(\mathfrak{sl}_{n-4}, f_{prin})$-module is also an ordinary module. Take a homogeneous vector $w \in H_{\theta}(U)$, let $M$ be the $L_{-1}(\mathfrak{sl}_{n-2})\otimes \Pi \otimes \cW_{\ell}(\mathfrak{sl}_{n-4}, f_{prin})$-submodule generated by $w$, and let $M_1$ (resp. $M_2$, $M_3$) be the subspace spanned by the elements obtained by applying the operators induced from $L_{-1}(\mathfrak{sl}_{n-2}) \otimes {\bf 1} \otimes {\bf 1}$ (resp. ${\bf 1} \otimes \Pi \otimes {\bf 1}$, ${\bf 1} \otimes {\bf 1} \otimes \cW_{\ell}(\mathfrak{sl}_{n-4}, f_{prin})$) on $w$. Obviously $M_1$ (resp. $M_2$, $M_3$) is an ordinary $L_{-1}(\mathfrak{sl}_{n-2})$ (resp. $\Pi$, $\cW_{\ell}(\mathfrak{sl}_{n-4}, f_{prin})$)-module. Moreover, $h(0)$ acts on $M_2$ semisimply as the Cartan subalgebra $\frh^{\natural}$ of $\frg^{\natural} = \gl_{n-2}$ acts semisimply on $H_{\theta}(U)$ (see \cite[Thm.~7.1]{KW}). Note that the category $\cO_{-1, \mathrm{ord}}(\mathfrak{sl}_{n-2})$ of ordinary modules for $L_{-1}(\sl_{n-2})$ is semisimple (Theorem \ref{levels}(3)), the category of ordinary $\Pi$-modules with semisimple $h(0)$ action for a generator $h$ of $\Pi$ is semisimple (\cite[Thm.~1.7.3]{FLM}), and the category of ordinary $\cW_{\ell}(\mathfrak{sl}_{n-4}, f_{prin})$-modules is semisimple because $\cW_{\ell}(\mathfrak{sl}_{n-4}, f_{prin})$ is rational and $C_2$-cofinite for $\ell = -h^{\vee} + \frac{n-3}{n-2}$ (\cite{Ar2}), thus $M_i$ are all direct sums of irreducible modules. It follows from \cite[Thm.~4.7.2]{FHL} that $M_1 \otimes M_2 \otimes M_3$ is a direct sum of irreducible $L_{-1}(\mathfrak{sl}_{n-2})\otimes \Pi \otimes \cW_{\ell}(\mathfrak{sl}_{n-4}, f_{prin})$-modules. Furthermore, there is a surjective $L_{-1}(\mathfrak{sl}_{n-2})\otimes \Pi \otimes \cW_{\ell}(\mathfrak{sl}_{n-4}, f_{prin})$-module map from $M_1 \otimes M_2 \otimes M_3 \rightarrow M$ that sends $a_1w \otimes a_2w \otimes a_3 w \mapsto a_1a_2a_3w$, where $a_1$ (resp. $a_2$, $a_3$) is an operator induced from $L_{-1}(\mathfrak{sl}_{n-2}) \otimes {\bf 1} \otimes {\bf 1}$ (resp. ${\bf 1} \otimes \Pi \otimes {\bf 1}$, ${\bf 1} \otimes {\bf 1} \otimes \cW_{\ell}(\mathfrak{sl}_{n-4}, f_{prin})$). Then as a quotient of $M_1\otimes M_2\otimes M_3$, $M$ is also a direct sum of irreducible $L_{-1}(\mathfrak{sl}_{n-2})\otimes \Pi \otimes \cW_{\ell}(\mathfrak{sl}_{n-4}, f_{prin})$-modules, and therefore $H_{\theta}(U)$ is a direct sum of irreducible $L_{-1}(\mathfrak{sl}_{n-2})\otimes \Pi \otimes \cW_{\ell}(\mathfrak{sl}_{n-4}, f_{prin})$-modules.

Let $\cC$ be the Deligne product of the category $KL_{-1}(\sl_{n-2})$, the category of ordinary $\Pi$-modules with semisimple $h(0)$-actions and the category of ordinary modules for $\cW_{\ell}(\mathfrak{sl}_{n-4}, f_{prin})$. It is clear that $\cC$ is a semisimple braided tensor category, and so its direct sum completion $\cC^{\oplus}$ also has a braided tensor category structure. From Theorem \ref{CL} and the fusion rules of $L_{-1}(\sl_{n-2})$-modules, which will be proved in \eqref{fusionrules}, $A = \cW_{-2}(\sl_n, \theta)$ is a fixed-point free simple current extension in the category $\cC^{\oplus}$. Let $\cF_A: \cC^{\oplus} \rightarrow (\cC^{\oplus})_A$ be the induction functor and $\cG_A: (\cC^{\oplus})_A \rightarrow \cC^{\oplus}$ be the restriction functor.

Suppose $\cG_A(H_{\theta}(U)) = \bigoplus_{i \in I}W_i$, where $W_i$ are irreducible $L_{-1}(\mathfrak{sl}_{n-2})\otimes \Pi \otimes \cW_{\ell}(\mathfrak{sl}_{n-4}, f_{prin})$-modules.
Since $A$ is a fixed-point free simple current extension, $\cF_A(W_i)$ are irreducible by Proposition \ref{prop:simplicity} (cf. \cite[Prop.~4.4]{CKM1}). By Frobenius reciprocity (Theorem \ref{sum1}(3)), there are nonzero $\cW_{-2}(\sl_n, \theta)$-module maps $\cF_A(W_i) \rightarrow H_{\theta}(U)$, and thus $\cF_A(W_i)$ are submodules of $H_{\theta}(U)$, so is their sum $\sum_{i \in I}\cF_A(W_i)$. Also because $H_{\theta}(U)$ is contained in $\sum_{i \in I}\cF_A(W_i)$, $H_{\theta}(U) = \sum_{i \in I}\cF_A(W_i)$, and therefore $H_{\theta}(U) = \bigoplus_{i\in J}\cF_A(W_i)$ for a subset $J \subset I$. On the other hand, $H_{\theta}(U)$ is a highest weight module, so it has to be irreducible.

For $n = 5$, Proposition \ref{prop:AKMPP3} and the fusion rules \eqref{fusionrules} show that $\cW_{-2}(\sl_5, \theta)$ is a simple current extension of $L_{-1}(\sl_3)\otimes M(3/5)$, fixed-point freeness of simple currents is clear since non-trivial Fock modules do not have fixed-points, then similar arguments as the case $n \geq 6$ show that the highest weight $\cW_{-2}(\sl_5, \theta)$-module $H_{\theta}(U)$ for every highest weight module $U$ in $KL_{-2}(\sl_5)$ is irreducible, and therefore every highest weight module in $KL_{-2}(\sl_5)$ is irreducible. As a result, $KL_{-2}(\sl_5)$ is a semisimple braided tensor category.
\end{proof}

\subsection{Generic $k$ such that $k+h^{\vee} \in \QQ_{\geq 0}$}\label{sec:generic}
Let $k = -2+\frac{1}{p}$ for $p \geq 1$. The category $\cO_{k,{\rm ord}}(\sl_2)$ for the affine vertex operator algebra $L_{k}(\mathfrak{sl}_2)$ is semisimple with generalized Verma modules all irreducible (\cite{C1}). By Theorem \ref{KLcoinside}, the category $KL_k(\sl_2)$ is semisimple. By Corollary \ref{main}, we have
\begin{prop}\label{generic}
Let $k = -2+\frac{1}{p}$ for $p \geq 1$. Then $KL_k(\mathfrak{sl}_2)$ are braided tensor.
\end{prop}

\begin{remark}
Proposition \ref{generic} completes the construction of braided tensor category structure for $L_k(\sl_2)$ at all levels $k \in \CC$ except the critical level $k = -2$. The rigidity of this tensor category structure at $k = -2+\frac{1}{p}$ will be proved in the next section using the main results of \cite{M} (see Example \ref{vp} and also \cite{ACGY}).
\end{remark}

\section{\texorpdfstring{$KL_{-1}(\mathfrak{sl}_m)$}{} and some more Examples}
In this section, we study several examples and especially we will obtain fusion rules and rigidity in these examples. It actually turns out that most of the modules that we study are simple currents (with only one exception in Example \ref{vp}) and rigidity of simple currents is clear, see for example \cite[Sec.~2.3]{CGR}.
The method we used here is certain duality between orbifold vertex operator algebras and compact groups via tensor category structure developed in \cite{M} (see also \cite{CM} when the group is finite abelian). We first recall the main results in \cite{M}.

\subsection{Tensor structure of orbifold vertex algebra}\label{sec:orbifold}
In \cite{DLM1}, Dong, Li and Mason proved the following Schur-Weyl-duality-type result for a simple vertex operator algebra $V$: if $G$ is a compact Lie group of automorphisms acting continuously on $V$, then $V$ is semisimple as a module for $G \times V^G$, where $V^G$ is the vertex operator subalgebra of $G$-fixed points. In particular,
\[
V = \bigoplus_{\chi \in \widehat{G}}M_{\chi}\otimes V_{\chi},
\]
where the sum runs over all finite-dimensional irreducible characters of $G$, $M_{\chi}$ is the finite dimensional irreducible $G$-module corresponding to $\chi$, and the $V_{\chi}$ are nonzero, distinct, irreducible $V^G$-modules. This decomposition thus sets up a correspondence between the category ${\rm Rep}\; G$ of
finite-dimensional continuous $G$-modules and the semisimple subcategory $\mathcal{C}_V$ of $V^G$-modules generated by the $V_{\chi}$.

Under the assumption that $V$ is a simple abelian intertwining algebra and $V^G$ has a tensor category of modules including the $V_{\chi}$, McRae (\cite{M}) proved that there is a braided equivalence between $\mathcal{C}_V$ and an abelian $3$-cocycle modification (also called a twist) of ${\rm Rep}\; G$. In particular, $\mathcal{C}_V$ inherits the rigidity from ${\rm Rep}\; G$. As a result, if the fixed point vertex subalgebra $V^G$ satisfies the assumptions in Theorem \ref{maintheorem}, then we will have a rigid braided tensor category $\cC_V$.
\begin{cor}\label{orbten}
Assume that $V^G$ is a vertex operator algebra such that all its ordinary irreducible modules are $C_1$-cofinite and all its generalized Verma modules are of finite length. If the category $\cC_V$ is contained in the category of finite length generalized modules for $V^G$, then $\cC_V$ has a rigid braided tensor structure and is braided equivalent to a twist of ${\rm Rep}\; G$.
\end{cor}

\subsection{ $KL_{-1}(\mathfrak{sl}_m)$}\label{sec: rigid}
Affine vertex algebras $L_{-1}(\mathfrak{sl}_m)$ for $m \geq 3$ (respectively, $m = 2$) are very nice examples of affine vertex operator algebras at collapsing levels (respectively, rational generic levels). As a consequence of Corollary \ref{collapsinglevels} and Proposition \ref{generic}, the category $KL_{-1}(\sl_m)$ of finite length generalized modules for $L_{-1}(\mathfrak{sl}_m)$, has a braided tensor category structure. Moreover, by Theorem \ref{KLcoinside}, the category $KL_{-1}(\sl_m)$ is semisimple.

Note that the vertex operator algebras $L_{-1}(\sl_m)$ appear as Heisenberg cosets of certain free field algebras and thus are amenable to our techniques. Using Corollary \ref{orbten}, we recover the fusion rules in a direct way and also prove the rigidity of the braided tensor category $KL_{-1}(\sl_m)$.

We first recall the coset realization of the vertex operator algebra $L_{-1}(\mathfrak{sl}_m)$ in the free fields. Recall the Weyl algebra (also called $\beta\gamma$ system) is an associative algebra with generators $\beta^i(n), \gamma^i(n)$ for $i = 1, \dots, m$ and $n \in \ZZ$ subject to the relations
\[
[\beta^i(n), \gamma^j(k)] = \delta_{i,j}\delta_{n+k,0},
\]
\[
[\beta^i(n), \beta^j(k)] = [\gamma^i(n), \gamma^j(k)] = 0
\]
for $n,k\in \ZZ$, $i, j = 1, \dots, m$.

The $\beta\gamma$-system $\beta\gamma^{\otimes m}$ of rank $m$ is the simple module for the Weyl algebra generated by the vacuum vector ${\bf 1}$ such that
\[
\beta^i(n){\bf 1} = \gamma^i(n+1){\bf 1} = 0 \;\;\; (n \geq 0).
\]
As vector space,
\[
\beta\gamma^{\otimes m} \cong \CC[\beta^i(-n), \gamma^j(-k)|n \geq 1, k \geq 0].
\]
It is well-known that there is a strongly $\ZZ$-graded vertex algebra structure on the $\beta\gamma$-system $\beta\gamma^{\otimes m}$. Thus $\beta\gamma^{\otimes m}$ can be regarded as an abelian intertwining algebra associated to the abelian group $\ZZ$ with the trivial twist (\cite[Example~2.17]{M}). Also note that $\beta\gamma^{\otimes m}$ can be viewed a simple $\frac{1}{2}\ZZ$-graded vertex algebra (see \cite{K2}, cf. \cite{AP2}).

Let
\[
h_{\beta\gamma} = \sum_{i=1}^m \beta^i(-1)\gamma^i(0){\bf 1},
\]
and let $M_{h_{\beta\gamma}}(1)$ be the Heisenberg algebra of level 1 generated by $h_{\beta\gamma}$. Clearly, the zero mode $h_{\beta\gamma}(0)$ acts semisimply on $\beta\gamma^{\otimes m}$ and gives a $\ZZ$-grading on $\beta\gamma^{\otimes m}$:
\begin{equation}
\beta\gamma^{\otimes m} = \bigoplus_{s\in \ZZ}M_s,\;\;\; M_s = \{v \in \beta\gamma^{\otimes m}|h_{\beta\gamma}(0)v = sv\}.
\end{equation}

It is shown in \cite{AP2} and \cite{C1} (cf. \cite{CKLR}) that
\begin{equation}\label{defnC2}
C(m) \overset{\tiny\mbox{def}}{=} {\rm Com}(M_{h_{\beta\gamma}}(1), \beta\gamma^{\otimes m}) = W_{-1}(\lambda_0)
\end{equation}
and for $s \in \ZZ$
\[
M_s \cong W_{-1}(\lambda_s)\otimes M_{h_{\beta\gamma}}(1,s),
\]
where as an $L_{-1}(\mathfrak{sl}_m)$-module
\begin{equation}\label{defnC2M}
W_{-1}(\lambda_s) =
\begin{cases}
L_{-1}(\lambda_s)\;\;\;\;&\text{if}\;\;\; m \geq 3\\
\bigoplus\limits_{i=0}^{\infty}L_{-1}(\lambda_{|s|+2i})\;\;\;\;&\text{if}\;\;\; m =2,
\end{cases}
\end{equation}
with $\lambda_s = s\omega_1$, $\lambda_{-s} = s\omega_{m-1}$ for $s \in \ZZ_{\geq 0}$.

Let $\mathcal{C}(m)$ be the category of $C(m)$-modules that are finite direct sums of $W_{-1}(\lambda_s)$, $s\in \ZZ$. If $m \geq 3$, from \cite[Prop.~4.1]{AP2} (cf. \cite{AP1}), $C(m) = L_{-1}(\mathfrak{sl}_m)$ and the set $\{L_{-1}(\lambda_s)|s \in \ZZ\}$
exhausts all the simple objects in the category $KL_{-1}(\mathfrak{sl}_m)$. Thus $\mathcal{C}(m) = KL_{-1}(\mathfrak{sl}_m)$. If $m=2$, note that $C(2)$ is a vertex operator algebra extension of $L_{-1}(\mathfrak{sl}_2)$,
the category $\mathcal{C}(2)$ is a full subcategory of
$$\left(KL_{-1}(\mathfrak{sl}_2)^{\oplus}\right)^{\text{loc}}_{C(2)},$$
that is the category of local modules for the commutative algebra object $C(2)$ in the direct sum completion  of $KL_{-1}(\mathfrak{sl}_2)$. Thus by Theorem \ref{sum1}(1), this category of local modules has braided tensor category structure. We will continue to show that the simple objects of $\cC(2)$ are simple currents, in particular, they are closed under fusion product, and hence $\cC(2)$ has a braided tensor category structure.

Consider the $U(1)$-action on $\beta\gamma^{\otimes m}$, we have $$(\beta\gamma^{\otimes m})^{U(1)} = W_{-1}(0)\otimes M_{h_{\beta\gamma}}(1,0) = C(m) \otimes M_{h_{\beta\gamma}}(1).$$
Let $\mathcal{M}$ be the category of $C(m)\otimes M_{h_{\beta\gamma}}(1)$-modules generated by $\{M_s|s\in \ZZ\}$. Since we know that the category $KL_{-1}(\sl_m)$ and the category generated by the highest weight $M_{h_{\beta\gamma}}(1)$-modules with real weights are braided tensor categories (\cite{CKLR}), the main results of \cite{M} (cf. Corollary \ref{orbten}) apply, it follows that $\mathcal{M}$ has a structure of symmetric tensor category and is braided equivalent to ${\rm Rep}(U(1))$. In particular, we have the fusion rules: for $s, t \in \ZZ$
\begin{equation}
M_s \boxtimes M_t = M_{s+t}.
\end{equation}
It follows from \cite[Prop.~3.3]{CKLR} that
\begin{prop}
The category $\mathcal{C}(m)$ has a rigid braided tensor category structure
with the following fusion rules:
\begin{equation}\label{fusionrules}
W_{-1}(\lambda_s)\boxtimes W_{-1}(\lambda_t) = W_{-1}(\lambda_{s+t}).
\end{equation}
\end{prop}

\begin{remark}
The algebra $C(2)$ is isomorphic to the rectangular $W$-algebra obtained from the affine vertex operator algebra of $\sl_4$ at level $-\frac{5}{2}$ via the quantum reduction corresponding to the nilpotent element of type $(2,2)$ (\cite[Cor.~5.3]{C1}, \cite[Thm.~8]{ACGY}).

It is easy to see that as $L_{-1}(\mathfrak{sl}_2)$-modules
\[
W_{-1}(\lambda_{s}) \cong W_{-1}(\lambda_{-s}).
\]
But we can tell from the fusion rules (\ref{fusionrules}) that as $C(2)$-modules, for $s \in \ZZ_{>0}$
\[
W_{-1}(\lambda_{s}) \ncong W_{-1}(\lambda_{-s}).
\]
\end{remark}

Because
\[
\beta\gamma^{\otimes m} = \bigoplus_{s \in \ZZ} M_s = \bigoplus_{s \in \ZZ}W_{-1}(\lambda_s)\otimes M_{h_{\beta\gamma}}(1,s),
\]
and the category $\mathcal{C}(m)$ is rigid, by \cite[Thm.~4.5]{CKM2}, there is a braid-reversed equivalence between $\mathcal{C}(m)$ and the category $\mathcal{F}_{\ZZ}$ of $M_{h_{\beta\gamma}}(1)$-modules generated by the simple objects $\{M_{h_{\beta\gamma}}(1,s)|s \in \ZZ\}$:
\begin{theorem}\label{thm:rigid-1}
The category $\mathcal{C}(m)$ is braid-reversed equivalent to the tensor subcategory $\mathcal{F}_{\ZZ}$ of the Heisenberg algebra $M_{h_{\beta\gamma}}(1)$.
\end{theorem}

As an application of the free field realization, we can compute the dimension of the objects $W_{-1}(\l_s)$ for $s \in \ZZ$:
\begin{lemma}\label{dimJ}
$\dim_{\cC(m)} W_{-1}(\l_s) = (-1)^s$.
\end{lemma}
\begin{proof}
We know that $\beta\gamma^{\otimes m}$ is a $\frac{1}{2}\ZZ$-graded vertex algebra, thus $\dim M_s = (-1)^s$ (\cite[Thm.~1.1]{CKL} and Remark \ref{sce}). Since $\dim M_{h_{\beta\gamma}}(1, s) = 1$, $\dim W_{-1}(\l_s) = (-1)^s$.
\end{proof}

\begin{remark}\label{sce}
The statement of \cite[Thm.~1.1]{CKL} only holds for a self-dual simple current, but in fact, the statement holds for any simple current as a consequence of \cite[Thm.~1.1, 1.3]{CKL} and Theorem \ref{sum1}(5): Let $\cA = \bigoplus_{s \in \ZZ}J^s$ be a simple current extension of a vertex operator algebra $V$ in a tensor category $\cC$ of certain $V$-modules, assume the ribbon twists satisfy $\theta_{J^s} = \pm 1$ and $\theta_{J^s} = \theta_{J^{s+2}}$, let
$\cA_0 = \bigoplus_{s \in \ZZ}J^{2s}$ and $\cA_1 = \bigoplus_{i \in \ZZ}J^{2s+1}$.
By \cite[Thm.~1.3]{CKL}, $\cA_0$ is a strongly $\ZZ$-graded vertex algebra. By Theorem \ref{sum1}(1), (5), the category $\cC_{\cA_0}^{loc}$ of local $\cA_0$-modules has a braided tensor category structure, and the induction functor $\cF_{\cA_0}: \cC \rightarrow \cC_{\cA_0}^{loc}$ is a braided tensor functor. Then it follows from \cite[Thm.~1.1]{CKL}, $\cA = \cA_0 \oplus \cA_1$ has a $\frac{1}{2}\ZZ$-graded vertex operator algebra structure if $c_{\cA_1, \cA_1} = 1$; and a vertex operator superalgebra structure if $c_{\cA_1, \cA_1} = -1$, where $c_{\cA_1, \cA_1}: \cA_1\btimes \cA_1 \rightarrow \cA_1 \btimes \cA_1$ is the braiding isomorphism. Note that $\cA_1 = \cF_{\cA_0}(J)$, since $\cF_{\cA_0}$ is a braided tensor functor, $c_{\cA_1, \cA_1} = c_{J, J}$. The dimension can be determined by the formula (see for example \cite[Thm.~1.2]{CKL})
\[
c_{J, J}\dim(J) = \theta_{J}.
\]
Note in our case $\cA = \beta\gamma^{\otimes m}$, $\cC = \cM$, $J = M_1$ with the ribbon twist $\theta_{J^s} = (-1)^s$, and by checking the weights and the dimensions of the weight spaces, $\cA_0$ is actually a vertex operator algebra.
\end{remark}

\subsection{More Examples}
In this subsection, we will use similar methods as in Subsection \ref{sec: rigid} to determine fusion rules and prove rigidity of certain full subcategories of more examples of $KL_k$.

\begin{exam}[\cite{C1, ACGY}]
Let $\mathcal{W}_p$ be the vertex operator algebra corresponding to the chiral algebra of Argyres-Douglas theories of type $(A_1, D_{2p})$ introduced in \cite{C1}. It is shown in \cite{ACGY} that $\cW_p$ is isomorphic to the algebra $\cR^{(p)}$ introduced in \cite{A}.

It was shown in \cite{C1} that $L_{-2+\frac{1}{p}}(\mathfrak{sl}_2) \otimes M_{\alpha}(1)$, where $(\alpha, \alpha) = -\frac{2}{p}$, embeds conformally in $\mathcal{W}_p$ and as $L_{-2+\frac{1}{p}}(\mathfrak{sl}_2) \otimes M_{\alpha}(1)$-modules
\[
\mathcal{W}_p \cong \bigoplus_{s \in \ZZ}M_p(s)\otimes M_{\alpha}(1,s),
\]
where
\[
M_p(s) := \bigoplus_{i=0}^{\infty}L_{-2+\frac{1}{p}}(\lambda_{|s|+2i}).
\]

It is clear that $M_p(0)$ is a commutative associative algebra object in $KL_{-2+\frac{1}{p}}(\mathfrak{sl}_2)^{\oplus}$. Set $$\mathcal{C}_p = (KL_{-2+\frac{1}{p}}(\mathfrak{sl}_2)^{\oplus})_{M_p(0)}^{loc}.$$ Then by Theorem \ref{sum1}(1), $\mathcal{C}_p$ has a vertex tensor category structure. Again by Corollary \ref{orbten}, the category generated by $\{M_p(s)\otimes M_{\alpha}(1,s)|s\in \ZZ\}$ is braided tensor equivalent to a twist of ${\rm Rep}(U(1))$. It follows from \cite[Prop.~3.3]{CKLR} that for $s,s'\in \ZZ$
\[
M_p(s)\boxtimes M_p(s') = M_p(s+s').
\]
As a result, the category of $M_p(0)$-modules generated by $\{M_p(s)|s\in\ZZ\}$ is a braided tensor category consisting of simple currents. Explicit construction and more are given in {\textup{\cite{ACGY}}}.
\end{exam}

\begin{exam}[\cite{A, ACGY}]\label{vp}
Let $V = \cV^{(p)}$ be the abelian intertwining algebra introduced in \cite{A} and $G = SU(2)$. As an $SU(2) \otimes L_k(\sl_2)$-module for $k = -2+\frac{1}{p}$, $p \geq 1$
$$\cV^{(p)} \cong  \bigoplus_{n=0} ^{\infty} E^{n\om_1} \otimes L_k(n\om_1).$$
By Proposition \ref{generic}, the category generated by $L_k(n\om_1)$ for $n \in \ZZ_{\geq 0}$, which is $KL_k(\sl_2)$, has a braided tensor category structure. Therefore the assumption in Corollary \ref{orbten} holds, $KL_k$ is braided equivalent to a twist of Rep $SU(2)$ and thus is rigid.
\end{exam}

The following three examples all follow from a decomposition theorem of certain conformal embeddings, that is  \cite[Thm.~5.1]{AKMPP4}.

\begin{exam} Let $m\geq 5$ and we consider $L_{-2}(\mathfrak{sl}_m)$. It follows from Theorem \ref{cor:beyondcollapsing} that the category $KL_{-2}(\sl_m)$ is a semisimple braided tensor category.
Let $\lambda = \sqrt{-2/m}$ and introduce the notation
\[
L_{-2}(n) := \begin{cases} L_{-2}(n\omega_2) & \qquad n\geq 0 \\ L_{-2}(-n\omega_{n-2}) &\qquad n<0 \end{cases}.
\]
Then by \cite[Thm.~5.1]{AKMPP4} as a module for $L_{-2}(\mathfrak{sl}_m) \otimes M(1, 0)$,
\[
L_{-2}(\mathfrak{so}_{2m}) \cong  \bigoplus_{n \in \mathbb Z} L_{-2}(n) \otimes M(1, n\lambda).
\]
Let $G=U(1)$ so that $L_{-2}(\mathfrak{so}_{2m})^G =  L_{-2}(\mathfrak{sl}_m) \otimes M(1, 0)$, which has a braided tensor category including the simple objects $L_{-2}(n) \otimes M(1, n\lambda)$, $n \in \ZZ$.
Applying Corollary \ref{orbten}, we get the fusion rules
\[
\left( L_{-2}(n) \otimes M(1, n\lambda) \right) \boxtimes \left( L_{-2}(r) \otimes M(1, r\lambda) \right) \cong L_{-2}(n+r) \otimes M(1, (n+r)\lambda)
\]
which imply \cite[Prop.~3.3]{CKLR}
\[
L_{-2}(n) \boxtimes L_{-2}(r) \cong L_{-2}(n+r)
\]
and rigidity as these are simple currents.
\end{exam}

\begin{exam} Consider $L_{-3}(\mathfrak{so}_{10})$. By Theorem \ref{levels}(5) and Corollary \ref{collapsinglevels}, the category $KL_{-3}(\so_{10})$ has a structure of braided tensor category.
Let  $\lambda = \sqrt{-1/4}$ and introduce the notation
\[
L_{-3}(n) := \begin{cases} L_{-3}(n\omega_5) & \qquad n\geq 0 \\ L_{-3}(-n\omega_{4}) &\qquad n<0 \end{cases}.
\]
Then by \cite[Thm.~5.1]{AKMPP4} as a module for $L_{-3}(\mathfrak{so}_{10}) \otimes M(1, 0)$,
\[
L_{-3}(\mathfrak{e}_{6}) \cong  \bigoplus_{n \in \mathbb Z}  L_{-3}(n) \otimes M(1, n\lambda).
\]
the same reasoning as in the previous example gives the fusion rules
\[
L_{-3}(n) \boxtimes L_{-3}(m) \cong L_{-3}(n+m)
\]
and rigidity as these are simple currents.
\end{exam}

\begin{exam} Consider $L_{-4}(\mathfrak{e}_6)$. By Theorem \ref{levels}(7) and Corollary \ref{collapsinglevels}, the category $KL_{-4}(\mathfrak{e}_6)$ has a structure of braided tensor category.
Let  $\lambda = \sqrt{-1/6}$ and introduce the notation
\[
L_{-4}(n) := \begin{cases} L_{-4}(n\omega_1) & \qquad n\geq 0 \\ L_{-4}(-n\omega_{6}) &\qquad n<0 \end{cases}.
\]
Then by \cite[Thm.~5.1]{AKMPP4} as a module for $L_{-4}(\mathfrak{e}_{6}) \otimes M(1, 0)$,
\[
L_{-4}(\mathfrak{e}_{7}) \cong  \bigoplus_{n \in \mathbb Z}  L_{-4}(n) \otimes M(1, n\lambda).
\]
the same reasoning as in the previous example gives the fusion rules
\[
L_{-4}(n) \boxtimes L_{-4}(m) \cong L_{-4}(n+m)
\]
and rigidity as these are simple currents.
\end{exam}

\section{Tensor category of \texorpdfstring{$L_1(\sl(n|m))$}{}}
\subsection{Free field realization of $L_1(\mathfrak{sl}(n|m))$}
In this subsection, we will describe the free field realization of $L_1(\mathfrak{sl}(n|m))$ and give a decomposition of $L_1(\mathfrak{sl}(n|m))$ as $C(m)\otimes L_1(\mathfrak{sl}_n)\otimes M(1)$-modules, where $M(1)$ is a rank one Heisenberg vertex operator algebra of level $1$.

Recall that the Clifford algebra is generated by $b^i(r), c^i(s)$ for $r, s \in \ZZ+1/2$, $i = 1, \dots, n$ with relations
\[
[b^i(r), c^j(s)] = \delta_{i,j}\delta_{r+s, 0},
\]
\[
[b^i(r), b^j(s)] = [c^i(r), c^j(s)] = 0.
\]
Note that the commutators are anticommutators because $b^i(r)$, $c^i(s)$ are odd for every $i, r, s$.

The $bc$-system $bc^{\otimes n}$ of rank $n$ is generated by the vacuum vector under the relations
\[
b^i(r){\bf 1} = c^j(r){\bf 1} = 0 \;\;\; (r \geq 0).
\]
As vector space,
\[
bc^{\otimes n} \cong \bigwedge(\{b^i(r), c^j(s)|r,s < 0\}).
\]

Let
\[
h_{bc} = \sum_{i=1}^n b^i(-1/2)c^i(-1/2){\bf 1},
\]
and let $M_{h_{bc}}(1)$ be the Heisenberg algebra of level 1 generated by $h_{bc}$. Clearly, the zero mode $h_{bc}(0)$ acts semisimply on $bc^{\otimes n}$ and gives a $\ZZ$-grading on $bc^{\otimes n}$:
\begin{equation}
bc^{\otimes n} = \bigoplus_{t\in \ZZ}N_t,\;\;\; N_t = \{v \in bc^{\otimes n}|h_{bc}(0)v = tv\}.
\end{equation}
It is well-known (see for example \cite{CKLR}) that ${\rm Com}(M_{h_{bc}}(1), bc^{\otimes n})$ is trivial for $n = 1$ and isomorphic to $L_{1}(\mathfrak{sl}_n)$ for $n \geq 2$. Moreover, for $t \in \ZZ$,
\[
N_t = L_1(\omega_{\overline{-t}}) \otimes M_{h_{bc}}(1,t),
\]
where $\ov{t}$ is the residue of $t$ modulo $n$.

Let $M_h(1)$ be the Heisenberg algebra generated by
\[
h = h_{bc} + h_{\beta\gamma}.
\]
The following free field realization of $L_{1}(\mathfrak{sl}(n|m))$ is proved in \cite{CKLR}:
\begin{theorem}[\cite{CKLR}]\label{CKLR}
For all $n,m\geq 1$, the commutant ${\rm Com}(M_h(1), \beta\gamma^{\otimes m}\otimes bc^{\otimes n})$ is isomorphic to the simple affine vertex superalgebra $L_{1}(\mathfrak{sl}(n|m))$.
\end{theorem}
Using Theorem \ref{CKLR}, we can construct $L_{1}(\mathfrak{sl}(n|m))$ as a simple current extension of $C(m)\otimes L_1(\mathfrak{sl}_n)\otimes M_y(1)$, where the generator $y$ is defined below.

{\bf Case 1: $n \neq m$.}\\
Let $M_x(1)$ be the Heisenberg generated by
\[
x = mh_{bc} + nh_{\beta\gamma}.
\]
The $M_{h_{\beta\gamma}}(1)\otimes M_{h_{bc}}(1)$-module $M_{h_{\beta\gamma}}(1,s)\otimes M_{h_{bc}}(1,t)$ can be regarded as the $M_h(1)\otimes M_x(1)$-module $M_h(1,s+t)\otimes M_x(1, ns+mt)$. Using this fact, we have
\begin{equation}\nonumber
\begin{split}
 {\rm Com}(M_h(1), \beta\gamma^{\otimes m}\otimes bc^{\otimes n})
&=  {\rm Com}\bigg(M_h(1), \bigoplus_{s,t\in \ZZ}W_{-1}(\lambda_s)\otimes L_1(\omega_{\overline{-t}}) \otimes M_{h_{\beta\gamma}}(1,s)\otimes M_{h_{bc}}(1,t)\bigg)\\
&=  {\rm Com}\bigg(M_h(1), \bigoplus_{s,t\in \ZZ}W_{-1}(\lambda_s)\otimes L_1(\omega_{\overline{-t}}) \otimes M_h(1,s+t)\otimes M_x(1, ns+mt)\bigg)\\
&=  \bigoplus_{s\in \ZZ}W_{-1}(\lambda_s)\otimes L_1(\omega_{\overline{s}}) \otimes M_{x}(1, (n-m)s).
\end{split}
\end{equation}
Let $M_y(1)$ be the level 1 Heisenberg vertex operator algebra generated by
\[
y: = -\frac{1}{\sqrt{mn(m-n)}}x,
\]
and denote the Fock module of weight $\lambda \in \CC$ by $M_y(1, \lambda)$. Then
\[
M_{x}(1, (n-m)s) \cong M_y(1, \mu s),
\]
where $\mu = \sqrt{\frac{m-n}{mn}}$. Thus we proved the following result:
\begin{prop}\label{voaext1}
As vertex operator superalgebras,
\[
L_{1}(\mathfrak{sl}(n|m))\cong \bigoplus_{s\in \ZZ}W_{-1}(\lambda_s)\otimes L_1(\omega_{\overline{s}}) \otimes M_y(1, \mu s).
\]
\end{prop}

{\bf Case 2: $n=m$.}\\
In this case, we choose
\[
x = h_{\beta\gamma}-h_{bc},
\]
let $M_{h,x}(1)$ be the rank 2 Heisenberg vertex operator algebra of level 1 generated by $h$ and $x$.

The $M_{h_{\beta\gamma}}(1)\otimes M_{h_{bc}}(1)$-module $M_{h_{\beta\gamma}}(1,s)\otimes M_{h_{bc}}(1,t)$ can be identified as the rank two Heisenberg algebra $M_{h,x}(1)$-module $M_{h,x}(1, s+t, s-t)$, the irreducible $M_{h,x}(1)$-module generated by $v_{s+t, s-t}$ with $h$, $x$ acting by scalar multiplications by $s+t$ and $s-t$, respectively. Using the same argument as in case $1$, we have
\begin{equation*}
L_1(\sl(m|m)) =  {\rm Com}(M_h(1), \beta\gamma^{\otimes m}\otimes bc^{\otimes m}) = \bigoplus_{s\in \ZZ}W_{-1}(\lambda_s)\otimes L_1(\omega_{\overline{s}}) \otimes M_h(1) \otimes M^{\text{top}}_{x}(1, 2s).
\end{equation*}
Here $M^{\text{top}}_{x}(1, 2s)$ denotes the $1$-dimensional top level of the Fock module. Note that $M_h(1)$ is a commutative Heisenberg subalgebra that is in fact central
in the sense that it commutes with the complete affine vertex superalgebra $L_1(\sl(m|m))$. Let $S = \bigoplus_{n = 1}^{\infty}M_h(1)_{(n)}$, where $M_h(1)_{(n)}$ are the $L(0)$-weight space with weight $n$ and let
\[
\jiao{S} : = {\rm Span}\{u_nw\; |\; u\in L_1(\sl(m|m)), w \in S, n \in \ZZ\}
\]
be the ideal generated by $S$, i.e. the smallest ideal of $L_1(\sl(m|m))$ containing $S$ (see \cite[Sec.~4.5]{LL}). The quotient $L_1(\sl(m|m))/\jiao{S} = L_{1}(\mathfrak{psl}(m|m))$. So that we get
\begin{prop}\label{decomp3}
As vertex operator superalgebras,
\[
L_{1}(\mathfrak{psl}(m|m))\cong \bigoplus_{s\in \ZZ}W_{-1}(\lambda_s)\otimes L_1(\omega_{\overline{s}}).
\]
\end{prop}

\begin{remark}
The decompositions in Proposition \ref{voaext1} and Proposition \ref{decomp3} was previously obtained in \cite[Prop.~4.1(7), Thm~4.4]{AMPP} when $m \geq 3$. The case $m=2$ of Proposition \ref{decomp3} was first obtained in \cite[Remark 9.11]{CG}.
\end{remark}

\subsection{Tensor category structure on $KL_1^{ss}(\mathfrak{sl}(n|m))$}
For a simple Lie superalgebra $\frg$ with a Cartan subalgebra $\frh$, it is natural study the full subcategory $KL_k^{ss}(\frg)$ of $KL_k(\frg)$ consisting of finite length generalized modules for $L_k(\frg)$ on which $h(0)$ act semisimply for all $h \in \frh$. The tensor structure on the category $KL_k^{ss}(\gl(1|1))$ has been studied in \cite{CMY2}.

In this subsection, we will study the detailed structure of the category $KL_1^{ss}(\mathfrak{sl}(n|m))$. Let's first study the case $n \neq m$. Let $\cC$ be the Deligne product of the category $\cC(m)$, $KL_1(\sl_n)$ and the category of ordinary $M_y(1)$-modules with semisimple $y(0)$-actions, it is clear that $\cC$ is a semisimple braided tensor category and so is its direct sum completion $\cC^{\oplus}$.
By Proposition \ref{voaext1}, the vertex operator superalgebra $L_1(\mathfrak{sl}(n|m))$ is a simple current extension of the vertex operator algebra $C(m)\otimes L_{1}(\mathfrak{sl}_n) \otimes M(1)$ in $\cC^{\oplus}$ (we omit the subscript $y$ in $M(1)$ for convenience).
Set $A = L_1(\sl(n|m))$, denote by $\mathcal{F}_A: \cC^{\oplus} \rightarrow (\cC^{\oplus})_A$ the induction functor and by $\mathcal G_A: (\cC^{\oplus})_A \rightarrow \cC^{\oplus}$ the restriction functor, i.e. for $X$ a $\mathcal C$-module $\mathcal F_A(X)$ is a not necessarily local $A$-module and for an $A$-module $Y$,  $\mathcal G_A(Y)$ is $Y$ viewed as an object in $\mathcal C^{\oplus}$.
One has
\[
\mathcal G_A\left(\mathcal F_A \left(X\right) \right) \cong   A \boxtimes X.
\]
By Theorem \ref{sum1}(5), $\mathcal{F}_A$ is a braided tensor functor if restricted to the subcategory of modules that centralize $A$, i.e. restricting to those modules that induce via $\cF_A$ to local $A$-modules. We first prove the following results:

\begin{lemma}\label{indirr}
If $X$ is a simple object in $\cC$, then $\cF_A(X)$ is a simple object in $\cC_{A}$.
\end{lemma}
\begin{proof}
For a simple object $X$ in $\cC$ of the form
\[
M_{s,a,\lambda} \overset{\tiny\mbox{def}}{=} W_{-1}(\lambda_s) \otimes L_1(\omega_a) \otimes  M_y(1, \lambda)
\]
for $s \in \ZZ$, $a \in \ZZ/n\ZZ$ and $\lambda \in \CC$, it follows from (\ref{fusionrules}) that for $i \in \ZZ$,
\[
M_{s, a, \lambda} \btimes M_{i, \overline{i}, \mu i} = M_{s+i, a+\overline{i}, \lambda + \mu i},
\]
thus
\[
M_{s, a, \lambda} \btimes M_{i, \overline{i}, \mu i} \ncong M_{s, a, \lambda} \btimes M_{j, \overline{j}, \mu j}
\]
unless $i = j$. Then it follows from Proposition \ref{prop:simplicity} that $\cF_A(X)$ is simple.
\end{proof}
Note that
\[
\mathcal G_A\left(\mathcal F_A \left(M_{s,a,\lambda}\right) \right) \cong   A \boxtimes M_{s,a,\lambda} \cong \bigoplus_{n \in \mathbb Z} M_{s+n,a+n,\lambda+\mu n}.
\]
\begin{theorem}\label{kl}
Every object in $KL_1^{ss}(\sl(n|m))$ is isomorphic to a direct sum of induced objects. In particular, the category $KL_1^{ss}(\sl(n|m))$ is semisimple such that all simple objects are induced objects from simple objects in $\cC$.
\end{theorem}
\begin{proof}
The proof is similar to the one of Theorem \ref{cor:beyondcollapsing}, we briefly sketch it here.

Let $X$ be an object in $KL_1^{ss}(\sl(n|m))$, we first show that $\cG_A$ maps an object in $KL_1^{ss}(\sl(n|m))$ to $\cC^{\oplus}$, i.e. $X \in KL_1^{ss}(\sl(n|m))$ regarded as a $C(m)\otimes L_{1}(\mathfrak{sl}_n) \otimes M(1)$-module is completely reducible and each simple summand is an object in $\cC$. From Theorem \ref{KLcoinside}, grading-restricted generalized modules for $C(m) = L_{-1}(\sl_m)$, $m \geq 3$ are completely reducible with simple summand in $\cC(m)$, we defer the proof for the algebra $C(2)$ in the following lemma as it is slightly different.
Also, the grading-restricted generalized $M(1)$-modules with semisimple $y(0)$ action are completely reducible (see for example \cite[Thm.~1.7.3]{FLM}), the same statement holds for $L_1(\sl_n)$ because it is rational and $C_2$-cofinite. Note that $X$, regarded as a $C(m)\otimes L_{1}(\mathfrak{sl}_n) \otimes M(1)$-module, is a grading-restricted generalized module, then it follows from the proof of \cite[Prop.~2.7]{DMZ} (see also Theorem \ref{cor:beyondcollapsing}) that $X$ is a direct sum of irreducible $C(m)\otimes L_{1}(\mathfrak{sl}_n) \otimes M(1)$-modules, i.e.,
\[
\mathcal G_A(X) \cong \bigoplus_{i \in I}X_i,
\]
where $X_i$ are simple objects in $\cC$.

Choose a summand $X_i$. By Frobenius reciprocity \eqref{eqn:fro-rec}
\[
\text{Hom}_{\cC_A^{\oplus}}\left(\mathcal F_A(X_i), X \right) \cong \text{Hom}_{\cC^\oplus}\left(X_i, \mathcal G_A(X)\right)
\]
and so there is a non-zero homomorphism from $\mathcal F_A(X_i)$ to $X$. By the previous lemma, $\mathcal F_A(X_i)$ is simple and thus $\mathcal F_A(X_i)$ is a submodule of $X$. It follows that $X$ equals the sum of simple modules $\cF_A(X_i)$, and therefore $X$ also equals a direct sum $\bigoplus_{i \in J}\cF_A(X_i)$ for a subset $J \subset I$.
\end{proof}

\begin{lemma}\label{cor:c2}
Every grading-restricted generalized module for the algebra $C(2)$ is completely reducible, and each simple summand is in $\cC(2)$.
\end{lemma}
\begin{proof}
Let $X$ be a grading-restricted generalized module for $C(2)$ (recall the definition in \eqref{defnC2}, \eqref{defnC2M}). Then $X$ is also a grading-restricted generalized module if regarded as an $L_{-1}(\sl_2)$-module. From Theorem \ref{KLcoinside}, $X$ is completely reducible, i.e.,
\[
\cG(X) = \bigoplus_{i \in I} X_i,
\]
where $\cG: KL_{-1}(\sl_2)^{\oplus}_{C(2)} \rightarrow KL_{-1}(\sl_2)^{\oplus}$ is the restriction functor, $X_i$, $i \in I$ are irreducible $L_{-1}(\sl_2)$-modules. By Frobenius reciprocity, there is a nonzero $C(2)$-module map $\cF(X_i) \rightarrow X$, where $\cF: KL_{-1}(\sl_2)^{\oplus} \rightarrow KL_{-1}(\sl_2)^{\oplus}_{C(2)}$ is the induction functor.

Next we show that the induced objects $\cF(X_i)$ are in the category $\cC(2)$, in particular, they are completely reducible. We claim that as $L_{-1}(\sl_2)$-modules:
\begin{equation}\label{inductionL-1}
\cG(\cF(L_{-1}(\l_s))) = \cG(W_{-1}(\l_0))\btimes L_{-1}(\l_s) = \bigoplus_{\substack{t = -s\\ t \equiv s\; ({\rm mod}\; 2)}}^{s} \cG(W_{-1}(\l_t))
\end{equation}
for $s \in \ZZ_{\geq 0}$.
From the fusion products of $L_{-1}(\sl_2)$-modules determined in Example \ref{vp},
\begin{align*}
\cG(W_{-1}(\l_0))\btimes L_{-1}(\l_1) &= \bigg(\bigoplus_{n \geq 0}L_{-1}(\l_{2n})\bigg)\btimes L_{-1}(\l_1)\\
&= L_{-1}(\l_1) \oplus \bigoplus_{n \geq 1}(L_{-1}(\l_{2n-1})\oplus L_{-1}(\l_{2n+1}))\\
&= \cG(W_{-1}(\l_{-1})) \oplus \cG(W_{-1}(\l_1)),
\end{align*}
and
\begin{align*}
\cG(W_{-1}(\l_0))\btimes L_{-1}(\l_2) &= \bigg(\bigoplus_{n \geq 0}L_{-1}(\l_{2n})\bigg)\btimes L_{-1}(\l_2)\\
&= L_{-1}(\l_2) \oplus \bigoplus_{n \geq 1}(L_{-1}(\l_{2n-2})\oplus L_{-1}(\l_{2n}) \oplus L_{-1}(\l_{2n+2}))\\
&= \cG(W_{-1}(\l_{-2})) \oplus \cG(W_{-1}(\l_0)) \oplus \cG(W_{-1}(\l_2)).
\end{align*}

Now we assume \eqref{inductionL-1} holds for less than or equal to $s$, then
\begin{align*}
\cG(\cF(L_{-1}(\l_{s})\btimes L_{-1}(\l_{1})))
&= \cG(\cF(L_{-1}(\l_{s+1}) \oplus L_{-1}(\l_{s-1}))) \\
&= \cG(W_{-1}(\l_0))\btimes L_{-1}(\l_{s+1}) \oplus \cG(W_{-1}(\l_0))\btimes L_{-1}(\l_{s-1}) \\
&= \cG(W_{-1}(\l_0))\btimes L_{-1}(\l_{s+1}) \oplus \bigg(\bigoplus_{t = -s+1,\; t \equiv s-1\; ({\rm mod}\; 2)}^{s-1} \cG(W_{-1}(\l_t))\bigg),
\end{align*}
On the other hand,
\begin{align*}
\cG(\cF(L_{-1}(\l_{s})\btimes L_{-1}(\l_{1})))
& = \cG(W_{-1}(\l_0))\btimes (L_{-1}(\l_{s})\btimes L_{-1}(\l_{1}))\\
& = (\cG(W_{-1}(\l_0))\btimes L_{-1}(\l_{s}))\btimes L_{-1}(\l_{1})\\
& = \bigg(\bigoplus_{t = -s,\; t \equiv s\; ({\rm mod}\; 2)}^{s} \cG(W_{-1}(\l_t))\bigg) \btimes L_{-1}(\l_1)\\
& = \bigoplus_{t = -s,\; t \equiv s\; ({\rm mod}\; 2)}^{s} (\cG(W_{-1}(\l_{t+1}))\oplus \cG(W_{-1}(\l_{t-1})))\\
& = \bigoplus_{\substack{t = -s-1\\ t \equiv s+1\; ({\rm mod}\; 2)}}^{s+1} \cG(W_{-1}(\l_t))\oplus \bigg(\bigoplus_{t = -s+1,\; t \equiv s-1\; ({\rm mod}\; 2)}^{s-1} \cG(W_{-1}(\l_t))\bigg),
\end{align*}
By comparing the right hand sides of the above two equations, we prove that \eqref{inductionL-1} holds for $s+1$, and therefore the claim is proved.

Now we use Frobenius reciprocity \eqref{eqn:fro-rec}:
\begin{align*}
\Hom_{\cC(2)}(\cF(L_{-1}(\l_s)), W_{-1}(\l_t)) &= \Hom_{KL_{-1}(\sl_2)^{\oplus}}(L_{-1}(\l_s), \cG(W_{-1}(\l_t)))\\
&=\begin{cases}
\CC \;\;\; \mathrm{if} \; |t|\leq s, \; t \equiv s \; (\mathrm{mod}\; 2) \\
0 \;\;\; \mathrm{else},
\end{cases}
\end{align*}
thus there is a surjective $C(2)$-module map
\[
\cF(L_{-1}(\l_s)) \rightarrow \bigoplus_{\substack{t = -s\\ t \equiv s \; (\mathrm{mod}\; 2)}}^s W_{-1}(\l_t).
\]
Combining with the $L_{-1}(\sl_2)$-module identification \eqref{inductionL-1}, we proved that as $C(2)$-modules
\[
\cF(L_{-1}(\l_s)) \cong \bigoplus_{\substack{t = -s\\ t \equiv s \; (\mathrm{mod}\; 2)}}^s W_{-1}(\l_t)
\]
for $s \in \ZZ_{\geq 0}$. The case for $s \in \ZZ_{<0}$ can be proved similarly.

Now we have shown that there is a surjective $C(2)$-module map $\bigoplus_{i \in I}\cF(X_i) \rightarrow X$ and each $\cF(X_i)$ is completely reducible, so $X$ is also completely reducible. Moreover, since each simple summand of the induced objects is in $\cC(2)$, $X$ must be in the direct sum completion $\cC(2)^{\oplus}$.
\end{proof}

In the above argument we have only used a few properties of the category $\cC$ and the algebra object $A$, namely that $\cC$ is a semisimple braided tensor category with simple currents $\{ A_i | i \in \mathbb Z\}$ such that
$A = \bigoplus_{i\in \mathbb Z} A_i$
is a fixed-point free super commutative algebra object in the direct sum completion.  In other words the following theorem is true and Theorem \ref{kl} and Theorem \ref{cor:beyondcollapsing} are both special cases of it.
\begin{theorem}\label{thm:semisimple}
Let $\cC$ be a semisimple braided tensor category and let $A$ be a simple current extension that is a fixed-point free super commutative algebra object in the direct sum completion of $\cC$.  Then the category of local modules of $A$ that lie in the direct sum completion of $\cC$ is semisimple and every object is isomorphic to an induced object.
\end{theorem}

Next we shall identify the simple objects in $KL_1^{ss}(\sl(n|m))$, from Theorem \ref{sum2}(5) and Theorem \ref{kl}, we only need to find the local induced modules from simple objects in $\cC$, which are of the form
\[
M_{s,a,\lambda} = W_{-1}(\lambda_s) \otimes L_1(\omega_a) \otimes M_y(1, \lambda)
\]
for $s \in \ZZ$, $a \in \ZZ/n\ZZ$ and $\lambda \in \CC$. The lowest conformal weights are
\[
h_{s,a,\lambda} = \frac{|s|(m+|s|)}{2m} +\frac{a(n-a)}{2n} +  \frac{\lambda^2}{2}.
\]

By Theorem \ref{sum1}(4), the induced module $\mathcal{F}_A(M_{s,a,\lambda})$ is local iff the monodromy $$\cM_{M_{1,\bar{1},\mu}, M_{s,a,\lambda}} = \one_{M_{1,\bar{1},\mu} \btimes M_{s,a,\lambda}},$$ then it follows from the balancing axiom that
\[
h_{s,a,\lambda} \equiv h_{s+1,\overline{a+1},\lambda+\mu}\;\;\; ({\rm mod}\;\; \ZZ).
\]
Thus $s,a,\lambda$ satisfy
\[
\frac{2s+1}{2m} -\frac{2a+1}{2n} + \frac{\mu^2}{2} + \lambda\mu \in \ZZ,
\]
i.e.,
\[
\lambda \mu - \frac{a}{n}+\frac{s}{m} \in \ZZ.
\]
Let $z \in \ZZ$ be such that $\lambda \mu + \frac{s}{m} = \frac{a}{n} + z$.

We now want to identify these modules as certain irreducible modules for the affine superalgebra $L_1(\mathfrak{sl}(n|m))$. For this purpose, we use the conventions and identifications of Section \ref{sec:slnm} and we introduce a convenient notation. Let $r \in \ZZ$ and set
\begin{equation}\label{eqn:convention}
L_1(r\omega_{n+m-1}) := \begin{cases} L_1(r\omega_{n+m-1}), &\quad r\geq 0 \\  L_1(-r\omega_{n+m-1})^*, & \quad r<0 \end{cases}
\end{equation}

We first show that the conformal gradings of the induced modules $\mathcal{F}_A(M_{a, s,\lambda})$ are bounded from below. i.e., $\{h_{s+t, \overline{a+t}, \lambda+\mu t}|t \in \ZZ\}$ have a lower bound. Moreover, we determine the lower bound as well as the top space. This is a straightforward computation, that we defer to the Appendix $7.2$.

\begin{prop}\label{simpleind}
Let $s \in \ZZ$, $a \in \ZZ/n\ZZ$ and $\lambda \in \mathbb R$ such that $\lambda \mu + \frac{s}{m} - \frac{a}{n} \in \ZZ$.
Then the induced modules are the highest-weight modules
\[
\mathcal{F}_A(M_{s,a,\lambda}) \cong L_1(n(\mu^2s-\lambda\mu)\omega_{n+m-1}).
\]
\end{prop}
\begin{proof}
If $s-a-zn\geq 0$, the top level space of $\mathcal{F}_A(M_{s,a,\lambda})$ as $\mathfrak{sl}_m \otimes\mathfrak{sl}_n\otimes  \mathfrak{h}$-module for a rank 1 abelian Lie algebra $\mathfrak{h}$ generated by the vector $y$ is
\begin{equation}\label{decomp1}
\bigoplus_{\tiny \begin{array}{c}(z+1)n-s \leq i\\-a< i \leq {n-a}\end{array}}E^{(s+i-(z+1)n)\omega_1}\otimes E^{\omega_{a+i}}\otimes \mathbb{C}_{\lambda+\mu i-\mu (z+1)n}.
\end{equation}
The highest weight of this module is $(\lambda+\mu (n-a)-\mu (z+1)n)\nu_n + (s+n-a-(z+1)n) \nu_{n+1}$, which equals $(s-a-zn)\omega_{n+m-1}=n(\mu^2s-\lambda\mu)\omega_{n+m-1}$ (see Section \ref{sec:slnm} for our conventions on weights), so that we conclude
\[
\mathcal{F}_A(M_{s,a,\lambda}) \cong L_1(n(\mu^2s-\lambda\mu)\omega_{n+m-1}).
\]
The case $s-a-zn\leq 0$ follows since the dual of $M_{s,a,\lambda}$ is $M_{-s,-a,-\lambda}$ and since the induction of the dual is the dual of the induced module (\cite[Lem.~1.16(2)]{KO}, cf. \cite[Prop.~2.77]{CKM1}).
\end{proof}

\begin{remark}
Consider the special case $\lambda \mu = -\frac{s}{m}$ and $a=0$, then we have
\[
\mathcal{F}_A(M_{s,0,- \frac{s}{\mu m}}) \cong L_1( s \omega_{n+m-1})
\]
and since the induction functor is monoidal we obtain the fusion rule
\begin{equation}
\begin{split}
 L_1( s \omega_{n+m-1})  \boxtimes L_1( s' \omega_{n+m-1})
\cong & \mathcal{F}_A(M_{s,0,- \frac{s}{\mu m}}) \boxtimes \mathcal{F}_A(M_{s',0,- \frac{s'}{\mu m}})
=  \mathcal{F}_A(M_{s,0,- \frac{s}{\mu m}} \boxtimes M_{s',0,- \frac{s'}{\mu m}}) \\
=& \mathcal{F}_A(M_{s+s',0,- \frac{s+s'}{\mu m}})
= L_1( (s+s') \omega_{n+m-1}).
\end{split}
\end{equation}
Note that the conformal weight $h_s$ of the top level of $L_1( s \omega_{n+1})$  is
\[
h_s = \frac{s(s+m-n)}{2(m-n)}.
\]
\end{remark}

As a consequence of Theorem \ref{kl} and Proposition \ref{simpleind},
\begin{cor}We have
\begin{itemize}
\item[(1)]The category $KL_1^{ss}(\sl(n|m))$ is semisimple with simple objects $L_1(r\omega_{n+m-1})$ (see the conventions \eqref{eqn:convention} and \eqref{omega}) for $r \in \ZZ$.
\item[(2)]The category $KL_1^{ss}(\sl(n|m))$ has a rigid braided tensor category structure with fusion rules
\[
L_1(r\omega_{n+m-1}) \btimes L_1(s\omega_{n+m-1}) = L_1((r+s)\omega_{n+m-1})
\]
for $r, s \in \ZZ$.
\end{itemize}
\end{cor}

\begin{remark}
If $n = m$, from Proposition \ref{decomp3}, the vertex operator superalgebra $L_1(\mathfrak{psl}(n|n))$ is a simple current extension of $C(n)\otimes L_1(\mathfrak{sl}_n)$.
Let
\[
M_{s,a} = W_{-1}(\lambda_s)\otimes L_1(\omega_a)
\]
be a simple $C(n)\otimes L_1(\mathfrak{sl}_n)$-module. By similar calculation $\mathcal{F}_A(M_{s,a})$ is local if and only if $s \equiv a$ $({\rm mod}\; n)$. Thus for any such simple objects $M_{s,a}$, $\mathcal{F}_A(M_{s,a}) = L_1(\mathfrak{psl}(n|n))$.
\end{remark}
\begin{cor} For $n\neq m$ and as an $M_h(1)\otimes L_1(\sl(n|m))$-module,
$$\beta\gamma^{\otimes m} \otimes bc^{\otimes n} \cong \bigoplus_{r\in \ZZ} M_h(1, r) \otimes L_1(r\om_{n+m-1}).$$
\end{cor}
\begin{proof}
This statement follows from the decomposition of $\beta\gamma^{\otimes m} \otimes bc^{\otimes n}$ in $\mathcal C$:
\begin{equation}\nonumber
\begin{split}
\beta\gamma^{\otimes m} \otimes bc^{\otimes n} &\cong \bigoplus_{s,t\in \ZZ}W_{-1}(\lambda_s)\otimes L_1(\omega_{\overline{-t}}) \otimes M_h(1,s+t)\otimes M_x(1, ns+mt) \\
&\cong \bigoplus_{s,t\in \ZZ}W_{-1}(\lambda_s)\otimes L_1(\omega_{\overline{-t}}) \otimes M_h(1,s+t)\otimes M_y\left(1, -\frac{ns+mt}{mn\mu}\right) \\
&\cong \bigoplus_{r \in \ZZ} M_h(1,r) \otimes \bigoplus_{s \in \mathbb Z}  W_{-1}(\lambda_s)\otimes L_1(\omega_{\overline{s-r}}) \otimes  M_y\left(1, -\frac{ns+m(r-s)}{mn\mu}\right) \\
&\cong \bigoplus_{r \in \ZZ} M_h(1,r)   \otimes \bigoplus_{s \in \mathbb Z} M_{s, s-r,  -\frac{r}{n\mu} +s\mu}  \\
&\cong \bigoplus_{r \in \ZZ} M_h(1,r)   \otimes  \mathcal G_A\left(\mathcal F_A \left( M_{r, 0,  -\frac{r}{m\mu}}\right)\right)  \\
&\cong \bigoplus_{r \in \ZZ} M_h(1,r)   \otimes  \mathcal G_A\left(L_1(r\om_{n+m-1})\right),
\end{split}
\end{equation}
so that
$$\beta\gamma^{\otimes m} \otimes bc^{\otimes n} \cong \bigoplus_{r\in \ZZ} M_h(1, r) \otimes L_1(r\om_{n+m-1})$$
as $M_h(1)\otimes L_1(\sl(n|m))$-module.
\end{proof}

\section{Appendix}
\subsection{Change of basis for the root system of $\mathfrak{sl}(n|m)$}\label{sec:slnm}
For the Lie superalgebra $\mathfrak{sl}(n|m)$, there is no unique simple root system and we choose a distinguished simple root system, i.e. we will only have one odd positive simple root. This is easiest described by embedding it into $\mathbb Z^{n+m} = \mathbb Z \epsilon_1 \oplus \dots \oplus \mathbb Z \epsilon_n \oplus  \mathbb Z \delta_1 \dots \oplus \mathbb Z\delta_m$
with norm $\epsilon_i\epsilon_j = \delta_{i, j}$ and $\delta_i\delta_j = -\delta_{i, j}$ and elements $a_i \epsilon_1 + \dots a_n \epsilon_n + b_1\delta_1 + \dots b_m \delta_m$ of the root lattice satisfy $a_1+ \dots + a_n +b_1 + \dots + b_m=0$.
For example, we denote $\Delta_{\overline{0}}^+ \cup \Delta_{\overline{1}}^+$
as a system of positive roots, and denote  by $\Pi$ the corresponding set of positive roots and our choice is
\begin{equation}
\begin{split}
\Delta_{\overline{0}}^+ &= \{\epsilon_i-\epsilon_j, \delta_i-\delta_j|i < j\},
\qquad\qquad
\Delta_{\overline{1}}^+ = \{\epsilon_i-\delta_j\},\\
\Pi &=  \{ \alpha_1, \dots, \alpha_{n+m-1}\}, \\
  \alpha_i :&= \begin{cases}  \epsilon_i-\epsilon_{i+1}, & \quad i=1, \dots, n-1 \\
\epsilon_n-\delta_1, & \quad i=n \\  \delta_{m+n-i}-\delta_{m+n+1-i}, & \quad i=n+1, \dots, m+n-1 \end{cases},
\end{split}
\end{equation}
see for example \cite{K1} for a reference for the Lie superalgebra $\sl(n|m)$ and its root system.

Let us set $\epsilon:= \epsilon_1 + \dots + \epsilon_n$ and $\delta:= \delta_1+ \dots  + \delta_m$.
The corresponding fundamental weights are
\begin{equation}\nonumber
\begin{split}
\omega_i &= \frac{1}{n-m} \left( (n-m-i) (\epsilon_1 +\cdots \epsilon_i)  - i(\epsilon_{i+1} + \cdots + \epsilon_n -\delta_1 - \cdots - \delta_m) \right)\\
&= \frac{i}{n-m}(\delta-\epsilon) + \epsilon_1 + \dots + \epsilon_i.
\end{split}
\end{equation}
for $i = 1, \dots, n-1$,
\[
\omega_n = \frac{1}{n-m}(-m(\epsilon_1 + \cdots + \epsilon_{n}) + n(\delta_1 +\cdots + \delta_m)) = \frac{-m\epsilon+n\delta}{n-m} = \frac{n}{n-m}(\delta-\epsilon) + \epsilon
\]
and
\begin{equation}\label{omega}
\omega_{n+j} =  \frac{j}{n-m}(\delta-\epsilon)  + \delta_{m-j+1} + \dots +  \delta_m
\end{equation}
for $j = 1, \dots, m-1$.
The positive roots of the even subalgebra are expressed in terms of ours as
\[
\{ \alpha_1, \dots, \alpha_{n-1}, \mu\omega_n, \alpha_{n+1}, \dots, \alpha_{n+m-1} \},
\]
where $\mu = \sqrt{\frac{m-n}{mn}}$ as in Section 6.1 and the corresponding fundamental weights are
\begin{equation}
\begin{split}
\nu_i &= -\frac{i}{n}\epsilon + \epsilon_1 + \dots \epsilon_i = \omega_i - \frac{i}{n} \omega_n\\
\nu_n &= -\mu \omega_n \\
\nu_{n+j} &=  \frac{j}{m}\delta - (\delta_1 + \dots \delta_j) = \omega_{n+m-j} -\frac{j}{m} \omega_n.
\end{split}
\end{equation}
Thus weight labels translate as
\begin{equation*}
\begin{split}
\Lambda &= \sum_{i = 1}^{n+m-1}a_i\nu_i  = \sum_{i = 1}^{n+m-1}b_i\omega_i \qquad \text{with} \\
b_i &= a_i, \qquad
b_n = -\mu a_n -\frac{1}{n} \sum_{i=1}^{n-1} ia_i - \frac{1}{m} \sum_{j=1}^{m-1} ja_j\qquad \text{and} \qquad
b_{n+j} = a_{n+j}
\end{split}
\end{equation*}
for $i=1, \dots, n-1$ and $j=1, \dots, m-1$.

\subsection{Top level of the induced module}
\begin{lemma}\label{lemma}
The conformal gradings of the induced modules $\mathcal{F}(M_{a, s,\lambda})$ are bounded from below and the bound is given by
\begin{equation*}
\frac{n}{2}(z+1)^2 - n(\frac{a}{n}+\frac{1}{2}+z)(z+1) + \frac{a(n-a)}{2n} + \frac{s(s+m)}{2m} + \frac{\lambda^2}{2}
\end{equation*}
if  $-zn-a+s$ is nonnegative;
the bound is
\begin{equation*}
\frac{n}{2}z^2 - zn(\frac{a}{n}-\frac{1}{2}+z) + \frac{a(n-a)}{2n} + \frac{s(s-m)}{2m} + \frac{\lambda^2}{2}
\end{equation*}
if $-zn-a+s$ is nonpositive.
\end{lemma}

The proof will be done by analyzing two cases. Let $t = rn + \bar{t}$ for $0\leq \bar{t} < n$ and $r \in \ZZ$.

{\bf Case 1}:
If $s+t \geq 0$,
\begin{align*}
&h_{s+t, \overline{a+t}, \lambda+\mu t} \\
=& \frac{(\overline{a+t})(n-\overline{a+t})}{2n} + \frac{(s+t)(s+t+m)}{2m} + \frac{(\lambda + \mu t)^2}{2}\\
=& \frac{(\overline{a+t})(n-\overline{a+t})}{2n} + \frac{(s+rn+\overline{t})(s+rn+\overline{t}+m)}{2m} + \frac{(\lambda + \mu (rn+\overline{t}))^2}{2}\\
=& \bigg(\frac{\mu^2n^2}{2}+\frac{n^2}{2m}\bigg)r^2 + \bigg(n\big(\mu^2+\frac{1}{m}\big)\overline{t}+\frac{2ns+mn}{2m}+\lambda\mu n\bigg)r + \mbox{const}.
\end{align*}
Since $\mu^2+\frac{1}{m} = \frac{1}{n}$ is always positive, the conformal weights $\{h_{s+t, \overline{a+t}, \lambda+\mu t}|t \in \ZZ\}$ always have a lower bound.
The minimal values attain when $r$ is an integer that is closest to
\[
-\frac{\overline{t}}{n} - \frac{s}{m} - \frac{1}{2} - \lambda\mu = -z-\frac{\overline{t}+a}{n} - \frac{1}{2}.
\]
We list the minimal values for all $\overline{t}$ as follows:
\begin{itemize}
\item If $\overline{t} + a = 0$, i.e., $a = 0, \overline{t} = 0$, then $r = -z$ or $-z-1$.
\item If $1 \leq \overline{t} + a \leq n-1$, then $r = -z-1$.
\item If $\overline{t} + a = n$, then $r = -z-1$ or $-z-2$.
\item If $\overline{t} + a \geq n+1$, then $r = -z-2$.
\end{itemize}
The minimum are all the same (independent of $\overline{t}$):
\begin{equation}\label{min}
\frac{n}{2}(z+1)^2 - n(\frac{a}{n}+\frac{1}{2}+z)(z+1) + \frac{a(n-a)}{2n} + \frac{s(s+m)}{2m} + \frac{\lambda^2}{2}
\end{equation}
It is easy to see \eqref{min} attains when $-zn-a+s$ is nonnegative, i.e. $\mu^2 s -\lambda\mu \geq 0$.

{\bf Case 2}:
If $s+t < 0$,
\begin{align*}
&h_{s+t,\overline{a+t}, \lambda+\mu t} \\
=& \frac{(\overline{a+t})(n-\overline{a+t})}{2n} + \frac{(s+t)(s+t-m)}{2m} + \frac{(\lambda + \mu t)^2}{2}\\
=& \bigg(\frac{\mu^2n^2}{2}+\frac{n^2}{2m}\bigg)r^2 + \bigg(n\big(\mu^2+\frac{1}{m}\big)\overline{t}+\frac{2ns-mn}{2m}+\lambda\mu n\bigg)r + \mbox{const}.
\end{align*}
The minimal values attain when $r$ is an integer that is closest to
\[
-\frac{\overline{t}}{n} - \frac{s}{m} + \frac{1}{2} - \lambda\mu = -z-\frac{\overline{t}+a}{n} + \frac{1}{2}.
\]
We list the minimal values for all $\overline{t}$ as follows:
\begin{itemize}
\item If $\overline{t} + a = 0$, i.e., $a = 0, \overline{t} = 0$, then $r = -z$ or $-z+1$.
\item If $1 \leq \overline{t} + a \leq n-1$, then $r = -z$.
\item If $\overline{t} + a = n$, then $r = -z-1$ or $-z$.
\item If $\overline{t} + a \geq n+1$, then $r = -z-1$.
\end{itemize}
The minimum are all the same (independent of $\overline{t}$):
\begin{equation}\label{min1}
\frac{n}{2}z^2 - zn(\frac{a}{n}-\frac{1}{2}+z) + \frac{a(n-a)}{2n} + \frac{s(s-m)}{2m} + \frac{\lambda^2}{2}.
\end{equation}
Similar to the previous case, \eqref{min1} attains when $-zn-a+s$ is nonpositive.

\bibliographystyle{alpha}

\end{document}